\documentclass[reqno,11pt]{amsart}
\calclayout
\usepackage{amsmath,amsthm,amssymb,amsbsy}
\usepackage{amsfonts,amstext,amscd}
\usepackage{babel}
\usepackage{graphicx}
\usepackage{bm}
\usepackage{commath,mathtools,setspace}
\usepackage{paralist}
\usepackage{mdwlist}
\usepackage{color}
\usepackage[T1]{fontenc}
\usepackage{geometry}
\usepackage[colorlinks=true,linkcolor=blue,citecolor=blue]{hyperref}
\usepackage[colorinlistoftodos,prependcaption,textsize=small]{todonotes}
\definecolor{red}{RGB}{255,0,0}
\definecolor{green}{RGB}{0,100,0}
\definecolor{blue}{RGB}{0,0,255}
\newtheorem{theorem}{Theorem}[section]

\newtheorem{lemma}[theorem]{Lemma}

\newtheorem{corollary}[theorem]{Corollary}
\newtheorem{proposition}[theorem]{Proposition}
\newtheorem{notation}[theorem]{Notation}

\newcommand{\sign}[1]{\mathop{\rm sign} \left(#1 \right)}
\theoremstyle{remark}
\newtheorem{remark}[theorem]{Remark}
\newtheorem{definition}[theorem]{Definition}
\newtheorem{example}[theorem]{Example}

\newcommand{\D}{\mathfrak{d}}

\newcommand{\lmesh}{\mathop{\rm lmesh}}
\newcommand*\pQq[6]{%
 {\ }_{#1}{\mathcal \phi}_{#2}{\left(\genfrac..{0pt}{}{#3}{#4};#5,#6\right)}%
}

\newcommand{\dil}[1]{\text{Dil}_{#1}}

\renewcommand{\P}{\mathbb{P}}

\newcommand{\R}{\mathbb{R}}
\newcommand{\C}{\mathbb{C}}
\newcommand{\Z}{\mathbb{Z}}

\newcommand{\N}{\mathbb{N}}

\newcommand{\qraising}[3]{\left(#1;#2\right)_{#3}}

\newcommand{\monicpols}{\mathbb P_n^*}
\newcommand{\monicpolreal}{\mathbb P_n^*(\mathbb R)}

\newcommand{\pols}{\mathbb P_n}

\newcommand{\nn}{\mathbb{N}}

\newcommand{\pp}{\mathbb{P}}

\newcommand{\rr}{\mathbb{R}}

\definecolor{apricot}{rgb}{0.98, 0.81, 0.69}

\definecolor{pastelblue}{rgb}{0.784, 0.902, 0.980}

\definecolor{greenpie}{rgb}{0.69, 0.95, 0.76}

\newcommand{\raising}[2]{\left(#1\right)^{\overline{#2}}}

\newcommand{\coef}[2]{e_{#1}\left(#2\right)}

\newcommand{\qsym}[1]{\mathrm{qSym}\left(#1\right)}

\newmuskip\pFqmuskip

\newcommand*\pFqN[6][8]{%
  \begingroup 
  \pFqmuskip=#1mu\relax
  \mathcode`\,=\string"8000
  \begingroup\lccode`\~=`\,
  \lowercase{\endgroup\let~}\pFqcomma
  {}_{#2}F_{#3}{\left(\genfrac..{0pt}{}{#4}{#5};#6\right)}%
  \endgroup
}
\newcommand{\pFqcomma}{\mskip\pFqmuskip}

\title[$q$-multiplicative and additive convolution]{Multiplicative and Additive  Finite Free Convolution For $q$-Polynomials}

\author[A. Mart\'{\i}nez-Finkelshtein]{Andrei Mart\'{\i}nez-Finkelshtein}

\address[AMF]{Department of Mathematics, Baylor University, TX, USA, and Department of Mathematics, University of Almer\'{\i}a, Spain}

\email{a\_martinez-finkelshtein@baylor.edu}

\author[R.~Morales]{Rafael Morales}

\address[RM]{Department of Mathematics, Baylor University, TX, USA}

\email{rafael\_morales2@baylor.edu}

\author[D.~Perales]{Daniel Perales}

\address[DP]{Department of Mathematics, University of Notre Dame, IN, USA}

\email{dperale2@nd.edu}

\date{\today}



\begin{document}

\begin{abstract}
We study $q$-analogs of finite free convolutions and their interaction with families of $q$-hypergeometric polynomials. First, we revisit the $q$-multiplicative finite free convolution, previously introduced in the literature, and show that it acts naturally on $q$-hypergeometric polynomials: the convolution of two such polynomials remains within the same class, with para\-meters obtained by concatenation. This observation provides a simple mechanism for constructing large families of $q$-hypergeometric polynomials whose zeros are real and whose logarithmic mesh is controlled. We illustrate it with an example of multiple little $q$-Jacobi polynomials of the first kind. A result of independent interest is also an alternative definition of the $q$-multiplicative convolution in terms of $q$-differential operators.

Motivated by the additive finite free convolution, we introduce a $q$-additive finite free convolution and study its algebraic and analytic properties. Although this convolution does not preserve real-rootedness in general, we show that a natural modification involving a $q$-multiplicative convolution restores the preservation of real roots and interlacing for polynomials with bounded logarithmic mesh.

Finally, we develop a systematic method to translate product identities of $q$-hypergeometric functions into convolution identities for $q$-hypergeometric polynomials. This approach yields several explicit formulas for $q$-additive convolutions and produces new families of real-rooted $q$-hypergeometric polynomials.
\end{abstract}

\maketitle

\tableofcontents

\section{Introduction}

Finite free convolution operations have become an important tool in the study of
real-rooted polynomials and spectral phenomena.
They were introduced in the work of Marcus, Spielman, and Srivastava
\cite{MR4408504} as part of their resolution of the Kadison--Singer problem.
These operations provide algebraic transformations of polynomials that preserve
real-rootedness and interlacing properties, and have since found applications in
spectral graph theory, random matrix theory, and the theory of hyperbolic
polynomials.
Closely related ideas had appeared earlier in the classical work of
Walsh \cite{Wal22} and Szeg\H{o} \cite{polya1914uber, szego1922} (see also \cite{MR2242036}), who studied transformations
preserving the location and interlacing of zeros of real polynomials.

Two fundamental operations in this context are the additive and multiplicative
finite free convolutions.
Both can be defined explicitly in terms of the coefficients of polynomials and
have strong stability properties with respect to real zeros.
In particular, the additive finite free convolution preserves real-rootedness
and interlacing, while the multiplicative convolution preserves real-rootedness
for polynomials whose zeros lie in suitable half-lines.
These results provide a powerful algebraic framework for constructing new
families of real-rooted polynomials from known ones, which was successfully exploited in \cite{MR4789095} in the framework of hypergeometric polynomials. Finite free convolution turned out to be a tool that allows us to derive conclusions about the zeros of these polynomials from the corresponding properties of more elementary ``building blocks.''

The main goal of this paper is to extend this machinery to the $q$-hypergeometric polynomials. The theory of basic hypergeometric series has
long provided a rich source of special functions and orthogonal polynomials
families.
Truncations of basic hypergeometric series lead to the so-called
$q$-hypergeometric polynomials; see, for instance,
\cite{koekoek2010hypergeometric}.
The location and interlacing properties of the zeros of these polynomials have
been investigated in a variety of settings; see, e.g.,
\cite{GJRS16, moak1981q}.
These results reveal a remarkable dependence of the zero distribution on the
parameters of the hypergeometric series and show that many such polynomials
have entirely real zeros under explicit parameter conditions.

However, the finite free convolutions, used as in \cite{MR4789095}, do not preserve the main properties of $q$-hypergeometric polynomials, and have to be adapted to this context. The first modification of this type of binary operations on polynomials, known as the
$q$-multiplicative finite free convolution and denoted by $\boxtimes_{n,q}$, was introduced
by Lamprecht \cite{lamprecht2016suffridge} and can be viewed as a deformation of the standard multiplicative finite convolution appearing in free probability.

We show that this operation interacts remarkably well with
$q$-hypergeometric polynomials. The starting point is the observation that the $q$-multiplicative convolution of two $q$-hypergeometric polynomials is again a $q$-hypergeometric polynomial whose parameters are obtained by concate\-nation of the parameter sets of the factors. This is a $q$-analog of the hypergeometric convolution structure
observed in \cite{MR4789095}. 
It provides a simple and systematic mechanism for constructing large families
of $q$-hypergeometric polynomials with controlled zero distribution.
In particular, by combining this operation with known real-rooted
$q$-hypergeometric families, we obtain broad classes of polynomials whose zeros
are real and whose relative spacing (or logarithmic mesh) is bounded. Moreover, we derive a simple equivalent definition of the $q$-multiplicative finite convolution in terms of some $q$-differential operators, an apparently new result that has an independent interest.

In the second part of the paper, we introduce a $q$-analog of the additive
finite free convolution, that is, a binary operation $\boxplus_{n,q}$ acting on polynomials of degree at most
$n$. It is defined in terms of the coefficients of the polynomial factors, but as in the previous case, it admits a natural representation in terms of $q$-differential
operators and can therefore be connected directly with the analytic structure
of basic hypergeometric series. It allows us to develop a systematic method to rewrite factorization formulas for 
$q$-hypergeometric functions as identities involving $q$-additive convolutions of
$q$-hypergeometric polynomials (again, as a $q$-analog of the similar identities de\-vel\-oped in \cite{MR4789095}). 
As illustrations, we derive several explicit formulas for $q$-additive
convolutions of $q$-hypergeometric polynomials based on the classical product
formulas such as Heine's transformation and related identities for basic
hypergeometric series.

The results of this paper provide new tools for constructing and analyzing
real-rooted families of \(q\)-hypergeometric polynomials. The
\(q\)-multiplicative finite free convolution \(\boxtimes_{n,q}\) fits especially
well with the \(q\)-hypergeometric structure: the convolution of two
\(q\)-hypergeometric polynomials is again \(q\)-hypergeometric, with parameters
obtained by concatenation. Combined with known real-rootedness and interlacing
criteria for basic hypergeometric polynomials, this gives a systematic way to
generate new families with controlled zero location and logarithmic mesh.

The \(q\)-additive convolution \(\boxplus_{n,q}\), introduced here as a
\(q\)-deformation of the additive finite free convolution, exhibits a different
phenomenon. In contrast with the classical operation, \(\boxplus_{n,q}\) does
not preserve real-rootedness in general. Our main result in this direction is
that this obstruction is removed by a universal \(q\)-multiplicative correction.
Namely, convolution with the fixed polynomial
\[
   {}_1\phi_1\!\left(\begin{matrix}q^{-n}\\ 0\end{matrix};q,-x\right),
\]
through \(\boxtimes_{n,q}\), restores real-rootedness for inputs in
\(\P_n^q(\mathbb R)\). In the positive case
\(p,r\in\P_n^q(\mathbb R_{>0})\), the corrected polynomial also has
positive zeros and satisfies the same logarithmic-mesh bound. We further obtain
an analogous preservation result for interlacing. In this sense, the corrected
\(q\)-additive convolution recovers, under natural logarithmic-mesh hypotheses,
the zero-preserving behavior characteristic of the classical additive finite
free convolution.

The paper is organized as follows.
In Section~\ref{sec:prelim}, we review basic notions concerning real-rooted polynomials,
interlacing, logarithmic mesh, and $q$-hypergeometric polynomials.
Section~\ref{sec:3} studies the $q$-multiplicative finite free convolution and its
action on $q$-hypergeometric polynomials, including the parameter
concatenation phenomenon described above.
In Section~\ref{sec:4}, we introduce the $q$-additive finite free convolution,
analyze its algebraic properties, and establish real roots preservation, as explained above. Finally, 
Section~\ref{sec:5} develops a general method for translating product
identities of basic hypergeometric functions into convolution identities
for $q$-hypergeometric polynomials and presents several illustrative
examples. Some technical proofs are deferred to Appendices~\ref{appendixA} and \ref{appendixB}.

\section{Preliminaries} \label{sec:prelim}

\subsection{Zero interlacing and mesh} \
 
Throughout this paper, $\P_n$ stands for the space of algebraic polynomials of degree $\le n$, and $\monicpols \subset \P_n$ is the subset of monic polynomials of degree $n$. For a set $K\subset \C$, we denote by $\P_n(K) $ (resp., $\monicpols(K)$) the subset of polynomials of degree $\le n$ (resp., monic polynomials of degree $n$)  with all their zeros in $K$. In particular, $\monicpolreal$ denotes the family of real-rooted monic polynomials of degree $n$, 
$\monicpols(\R_{\ge 0})$
is\footnote{\, Hereafter, $\R_{\ge 0}$ stands for $[0,+\infty)$, etcetera.} the subset of $\monicpolreal$ of polynomials having only non-negative roots, etcetera. 

We denote the zeros of a polynomial $p\in \P_n$, by $\lambda_1(p),\lambda_2(p), \dots, \lambda_n(p)$, accounting for multiple zeros. Moreover, if $p\in \P_n(\R)$ we assume that
\begin{equation}
\label{eq:zeros}
\lambda_1(p) \leq \dots \leq \lambda_n(p)  .  
\end{equation}

\begin{definition}[Interlacing] 
Let $p\in \P_n(\R)$ and $r\in \P_m(\R)$, with the zeros denoted as in \eqref{eq:zeros}.  
We say that $r$ \textbf{interlaces} $p$ (or, equivalently, that \textbf{zeros of $r$ interlace zeros of $p$}, see, e.g.,~\cite{MR3051099}), and denote it $p \preccurlyeq r$, if 
\begin{equation} \label{interlacing1}
    m=n \quad \text{and} \quad \lambda_1(p) \leq \lambda_1(r) \leq \lambda_2(p) \leq \lambda_2(r) \leq \cdots \leq  \lambda_n(p) \leq \lambda_n(r),
\end{equation}
or if
\begin{equation} \label{interlacing2}
    m=n-1 \quad \text{and} \quad \lambda_1(p) \leq \lambda_1(r) \leq \lambda_2(p) \leq \lambda_2(r) \leq \cdots \leq  \lambda_{n-1}(p) \leq \lambda_{n-1}(r) \le \lambda_n(p).
\end{equation}
Furthermore, we use the notation $p \prec r$ when all inequalities in \eqref{interlacing1} or \eqref{interlacing2} are strict. Notice that $\prec$ or $\preccurlyeq$ are not order relations, since they lack transitivity: from the fact that $p\prec r$ and $r\prec s$ it does not necessarily follow that $p\prec s$.
\end{definition}

\begin{lemma}[\cite{jordaan2009interlacing}]\label{Jordaaninterlacing}
Let $p_n\in\P_n(\R_{\geq 0})\setminus \P_{n-1}(\R_{\geq 0})$ and $r_{n-1}\in\P_{n-1}(\R_{> 0}) \setminus \P_{n-2}(\R_{> 0})$ such that $p_n\prec r_{n-1}$ and assume that a polynomial $g$ of degree $n$ satisfies, for a constant $c\neq 0$,
    $g(x)=p_n(x)+c \, r_{n-1}(x)$  or  $ g(x)=p_n(x)+c x\, r_{n-1}(x)$. 
    Then $g\in \P_n(\R)$ and
    $$
    p_n\prec g \prec r_{n-1} \quad \text{ or } \quad g\prec p_n \prec r_{n-1}.
    $$
\end{lemma}

The relative spacing of real zeros of a polynomial is given by the following notion: 
\begin{definition}\label{def:logarithmicmesh}
For $p\in \P_n(\R_{>0})$, $n\geq 2$, with the zeros denoted as in \eqref{eq:zeros}, 
its \textbf{logarithmic mesh} is 
\begin{equation*}
    \lmesh{p}:=\max_{1\leq j\leq n-1}\frac{\lambda_{j}(p)}{\lambda_{j+1}(p)}\leq 1.
\end{equation*}
For $p\in \P_n(\R_{<0})$ we understand by $\lmesh{p}$ the value $\lmesh{r}$, with $r(x)=p(-x)$.
\end{definition}

In the rest of the paper, the polynomials with an upper bound on the logarithmic mesh will play a prominent role, which motivates the following notation: for $q\in (0,1)$ and $K$ standing for either $\R_{>0}$ or $\R_{<0}$, 
$$\P^q_n(K):=\{r\in \P_n(K):\lmesh{r}< q\} \quad \text{and} \quad \overline{\P^q_n}(K):=\{r\in \P_n(K):\lmesh{r}\leq q\}.
$$
We can characterize $\P_n^q(\R_{> 0})$ as follows, see \cite[Lemma 23]{lamprecht2016suffridge}: 
\begin{align}
    p\in\P_n^q(\R_{> 0})& \Longleftrightarrow p(x)\prec p(qx)\label{lmeshRplus}.
\end{align}
A similar characterization is valid for 
$\overline{\P^q_n}(\R_{>0})$ when replacing $\prec$ by $\preccurlyeq$.
Finally, we will write
$$
p \propto r
$$
to indicate that polynomials $p$ and $r$ coincide up to a nonzero multiplicative constant.

\subsection{\texorpdfstring{$q$-hypergeometric polynomials}{q-hypergeometric polynomials}}\

For $q\in(0,1)$, $a\in\R$ and $k\in \nn$, the \textbf{$q$-factorial} (also known as the \textbf{$q$-Pochhammer symbol}) is defined by 
\begin{equation}\label{eq:def.qfactorial}
    \qraising{a}{q}{0}:=1, \qquad  \qraising{a}{q}{k}:=\prod_{j=0}^{k-1}(1-aq^{j}).
\end{equation}
One can retrieve the conventional Pochhammer symbol by letting $q$ tend to 1,
\begin{equation}\label{qraising3}
   \lim_{q\to 1}\frac{\qraising{q^a}{q}{k}}{(1-q)^k}= \raising{a}{k}:=a(a+1)\cdots(a+k-1).
\end{equation}

To deal with the product of several $q$-factorials, we extend the definition to vectors of parameters. For  $\bm a=(a_1, \dots,a_r)\in\R^r$, we denote
\begin{equation*}
    \qraising{\bm a}{q}{k}:=\prod_{j=1}^r \qraising{a_j}{q}{k}.
\end{equation*}

With this notation, the \textbf{$q$-hypergeometric series} for $\bm a=(a_1,\cdots,a_r)\in\R^r$ and $\bm b=(b_1,\cdots,b_s)\in\R^s$ is (see \cite[\S 1.10]{koekoek2010hypergeometric})\footnote{\, Here and in what follows we understand by $\binom{k}{2}$ the value $k(k-1)/2$. With this convention, $\binom{0}{2}=\binom{1}{2}=0$.}

\begin{equation}\label{def:qhyperserie}
    \pQq{r+1}{s}{a_0,\bm a}{\bm b}{q}{x}:=\sum_{k=0}^{\infty}\frac{\qraising{a_0}{q}{k}  \qraising{\bm a}{q}{k}}{\qraising{\bm b}{q}{k}}(-1)^{(s-r)k}q^{(s-r)\binom{k}{2}}\frac{x^k}{\qraising{q}{q}{k}}.
\end{equation}
If $a_j=q^{-n}$ for some $j$ and $n\in\N$, then the series truncates:
\begin{align}\label{def:qpolynomial1}
     \pQq{r+1}{s}{q^{-n},\bm a}{\bm b}{q}{x}:=\sum_{k=0}^{n}(-1)^{(s-r)k}q^{(s-r)\binom{k}{2}}\frac{\qraising{q^{-n}}{q}{k} }{\qraising{q}{q}{k}}\frac{ \qraising{\bm a}{q}{k}}{\qraising{\bm b}{q}{k}}\, x^k
\end{align}
is a $q$-\textbf{hypergeometric polynomial} of degree $\le n$, under the assumption that
\begin{equation}  \label{constrainB}
    b_1,\dots b_s \in \R \setminus \{ q^{-1},q^{-2},\dots, q^{-n+1} , q^{-n}\}.
\end{equation}

\subsection{\texorpdfstring{Real zeros of $q$-hypergeometric polynomials}{Real zeros of q-hypergeometric polynomials}}\label{propertiesqhyper} 

Location and interlacing of zeros for   
$$
p(x)=\pQq{i+1}{j}{q^{-n},a}{b}{q}{x}\qquad \text{ where }  i,j\in\{0,1\},
$$
has been studied in \cite{gochhayat2016interlacing, moak1981q, martinezfinkelshtein2025zerosorthogonallittleqjacobi}. We've used \cite{martinezfinkelshtein2025zerosorthogonallittleqjacobi} to compile conditions on the parameters $a$ and $b$ that guarantee that $p$ is real-rooted; see Table~\ref{tab:2Q1}. 
\begin{table}[h]
    \centering
\begin{tabular}{|c|c|c|c|c|c|}
\hline \textbf{Row} &  $(i,j)$ & $a$ & $b$ & \textbf{Roots in} & $\lmesh$ \\
\hline \hline 1 & $(1,1)$ & $\left(-\infty, bq^{n-1}\right)$ & $(0,1)$ & $(0,q)$ & $<q$ \\[1mm]
\hline 2 & $(1,1)$ & $\Big\{b,bq,\dots, bq^{n-1} \Big\} $ & $(0,1)$ & $(0,q)$ & $\leq q$ \\
\hline 3 & $(1,1)$ & $\Big(q^{1-n},\infty\Big)$ & $(-\infty,0)$ & $\R_{<0}$ & $<q$ \\
\hline 4 & $(1,1)$ & $\Big(q^{1-n},bq^{n+1}\Big)$ & $\Big(q^{-2n+2},\infty\Big)$ &  $\R_{>0}$  & $<q$ \\
\hline 5 & $(1,1)$ & $(-\infty,0)$ & $\{0\}$ & $(0,q)$ & $<q$\\
\hline 6 & $(1,1)$ &  $\Big(q^{1-n},\infty\Big)$  & $\{0\}$ & $\R_{<0}$  & $<q$ \\
\hline 7 & $(1,0)$ & $\Big(q^{1-n},\infty\Big)$ & $\cdot$ & {$\R_{> 0}$} & $<q$ \\
\hline 8 & $(0,1)$ & $\cdot$ & $(0,1)$ & $\R_{<0}$ & $<q^2$ \\
\hline 9 & $(0,1)$ & $\cdot$ & $0$ & $\R_{<0}$ & $\leq q^2$ \\
\hline 10 & $(0,1)$ & $\cdot$ & $(-\infty,0)$ & $\R_{<0}$ & $<q$ \\
\hline
\end{tabular}
\vspace{2mm}
\caption{Real zeros of some $\pQq{i+1}{j}{q^{-n}, \, a}{  b}{q}{x}$ polynomials, $i,j\in\{0,1\}$.
}
    \label{tab:2Q1}
\end{table}  

By \cite[Section 3.2]{martinezfinkelshtein2025zerosorthogonallittleqjacobi}, for $q^2\leq t_1,t_2<1$, the following interlacings hold, 
\begin{align}
    \pQq{1}{1}{q^{-n}}{t_1b}{q}{x}&\preccurlyeq\pQq{1}{1}{q^{-n}}{b}{q}{x} \quad \text{for } b\geq 0 ,\label{qinterlacing1}\\
     \pQq{1}{1}{q^{-n}}{b}{q}{x}&\preccurlyeq \pQq{1}{1}{q^{-n}}{t_1^{-1}b}{q}{t_1^{-1}x} \quad \text{for } b< 0 \label{qinterlacing2}.
\end{align}
For $0<b<1$,
\begin{align}
    \pQq{2}{1}{q^{-n},t_1a}{b}{q}{x}&\prec\pQq{2}{1}{q^{-n},t_2a}{t_2b}{q}{x} \quad \text{for } 0\leq a<bq^{n-1},\label{qinterlacing3}\\
    \pQq{2}{1}{q^{-n},t_1a}{b}{q}{x}&\preccurlyeq\pQq{2}{1}{q^{-n},t_2a}{t_2b}{q}{x} \quad \text{for } a\in \Big\{b,bq,\dots, bq^{n-1} \Big\} ,\label{qinterlacing6}\\
    \pQq{2}{1}{q^{-n},a}{b}{q}{x}&\prec\pQq{2}{1}{q^{-n},t_1a}{b}{q}{x} \quad \text{for } a<0,\label{qinterlacing4}\\
     \pQq{2}{0}{q^{-n},t_1^{-1}a}{\cdot}{q}{t_1x} &\prec \pQq{2}{0}{q^{-n},a}{\cdot}{q}{x} \quad \text{for } a> q^{1-n}.\label{qinterlacing5}
\end{align}
Equation \eqref{qinterlacing6} does not appear explicitly in \cite{martinezfinkelshtein2025zerosorthogonallittleqjacobi}, but it is a straightforward consequence of the results therein. 

\subsection{\texorpdfstring{$q$-Derivative}{q-Derivative}}

For $q\in (0,1)$, the \textbf{$q$-derivative} operator $\mathfrak{D}_q $ is defined as
\begin{equation*}
    \mathfrak{D}_q f(x) := \begin{cases} 
\dfrac{f(x) - f(qx)}{(1-q)x}, & x \neq 0, \\[8pt] 
f'(0), & x = 0.
\end{cases}
\end{equation*}
In this work, it is more convenient to use a renormalization of the \textbf{$q$-derivative} that will simplify our computations:
\begin{equation}\label{def:qderivative}
    \mathfrak{d}_q f(x) := (1-q)\,\mathfrak{D}_qf(x).
\end{equation}
From the definition we get for $c\in\C$,
\begin{equation}\label{qderivative1}
    \D_q[f(cx)]=c(\D_qf)(cx).
\end{equation}

It is a well known fact that if a polynomial has lmesh less than $q$, then so does its $q$-derivative, and these two polynomials interlace, see for instance \cite[theorem 25]{lamprecht2016suffridge}. In our notation, this is summarized as
$$p \in \P^q_n(\rr_{> 0}) \quad  \Rightarrow \quad \D_q p \in \P^q_{n-1}(\rr_{> 0}) \text{ and }  p\prec \D_q p.$$ 
A converse of this result will be useful to construct polynomials with lmesh less than $q$.
\begin{proposition}\label{prop:lmeshinterlaing}
    Let $p $ be a polynomial of degree $n$ with simple positive roots. Then 
    $$\D_q p\in \P_{n-1}(\rr_{> 0}) \text{ and } p\prec \D_q p \quad  \Rightarrow \quad  \lmesh p <q.$$
\end{proposition}
\begin{proof}
By hypothesis $p\prec \D_q p$ and from the definition of $\D_q$ we know that 
$$p(qx)=p(x)-x\D_qp(x).$$
By  Lemma \ref{Jordaaninterlacing}, we get that either $p(x)\prec p(qx)$ or $p(qx)\prec p(x)$. Since $p$ has positive roots then $p(x)\prec p(qx)$, which is equivalent to $p \in \P^q_n(\rr_{> 0})$ by \eqref{lmeshRplus}.
\end{proof}

\begin{remark}\label{remark:lmeshinterlaing}
    Note that Proposition \ref{prop:lmeshinterlaing} still holds if we replace the hypothesis of $p\prec \D_q p$ by $p\prec (\D_qp)(q^{-1}x)$. The proof is completely analogous, but using 
    $$
    p(qx)=p(x)+x(\D_q p)(q^{-1}x).
    $$
\end{remark}


\section{\texorpdfstring{$q$-Multiplicative Convolution}{q-Multiplicative Convolution}} \label{sec:3}

\subsection{Preliminary results}\

In this section, we will introduce definitions and results of $q$-multiplicative convolution, studied in \cite{lamprecht2016suffridge}.

In what follows, we will denote the coefficients of a polynomial $p$ of degree $n$ as follows:
\begin{equation}\label{monicP}
    p(x)=\sum_{j=0}^n x^{n-j}  e_j(p).
\end{equation}
Since we do not require a priori $e_0(p)\neq 0$, the notation $e_j(p)$ implicitly assumes the dependence on $n$. It is convenient to keep it in mind, since we will avoid mentioning it explicitly to simplify the notation.

\begin{definition}\label{def:qmultiplication}
Given two polynomials, $p$ and $r$, of degree at most $n$, the 
\textbf{$n$-th $q$-multiplicative finite free convolution} of $p$ and $r$, denoted by $p\boxtimes_{n,q} r$, is a polynomial of degree at most $n$, which, following the notational convention \eqref{monicP}, is defined as 
$$
[p\boxtimes_{n,q} r](x): =\sum_{k=0}^n  e_k(p\boxtimes_{n,q} r)x^{n-k},
$$
with 
\begin{equation}\label{def:qmulticoeff}
    e_k(p\boxtimes_{n,q} r)=\frac{\qraising{q}{q}{n-k}}{\qraising{q^{-n}}{q}{n-k}}\,  e_k(p) e_k(r).
\end{equation}
\end{definition}
This definition differs from the given in \cite{lamprecht2016suffridge} only by a rescaling of $x$ by $-q^{-n}$. This is done to simplify the notation when applied to $q$-hypergeometric polynomials. 

The $q$-multiplicative finite free convolution is a linear operator  from $\P_n \times \P_n$ to $\P_n$: if $p,r,g\in \P_n$,  and $\alpha\in\rr$, then 
\begin{equation*}
    (\alpha p+g)\boxtimes_{n,q} r= \alpha (p\boxtimes_{n,q} r)+g\boxtimes_{n,q} r.
\end{equation*}

The identity for the $q$-multiplicative convolution is given by the polynomial  
\begin{equation}\label{def:identitymultiplicaitve}
    E_{n,q}(x):=\prod_{k=1}^n \left( 1-q^{-k}x \right)=\pQq{1}{0}{q^{-n}}{\cdot}{q}{x},
\end{equation}
which has roots at $q,q^2,\dots q^n$. Indeed, one can check that $\coef{k}{E_{n,q}}=\frac{\qraising{q^{-n}}{q}{n-k}}{\qraising{q}{q}{n-k}}$ for $k=0,\dots,n$, implying by Definition \ref{def:qmultiplication} that
\begin{align}
    \label{identityMult1}
        p &\boxtimes_{n,q} E_{n,q} = p \qquad \text{ for all }p\in\pols.
\end{align}
It can be easily checked that scaling of the variable in $E_{n,q}$ before its $q$-multiplicative convolution has a simple effect on the result: for a scalar $\alpha$,

\begin{equation}
    \label{identityMult3}
    p(x) \boxtimes_{n,q} E_{n,q}(\alpha x) =  p\left(\alpha x\right).
\end{equation}
As a consequence,
\begin{align}
 \label{identityMult2}
        p(\alpha x)&\boxtimes_{n,q} r(\beta x)=p(x)\boxtimes_{n,q} r(\alpha\beta x).
\end{align}

\begin{remark}[Relation with multiplicative finite free convolution]
  Recall that for polynomials $p,r\in\P_n$ in the form \eqref{monicP}, we define the $n$-th multiplicative finite free convolution $p\boxtimes_n r$ by 
\begin{equation}\label{coeffMultiplicativeConv}
 e_k(p\boxtimes_n r)=  (-1)^k\binom{n}{k}^{-1} e_k(p) e_k(r), \quad k=0, 1, \dots, n. 
\end{equation}
It follows that
    $$
    \lim_{q\to 1}p \boxtimes_{n,q} r= p\boxtimes_n r.
    $$
Indeed, using the limit property \eqref{qraising3}, one has 
    $$
    \lim_{q\to 1}\frac{\qraising{q}{q}{n-k}}{\qraising{q^{-n}}{q}{n-k}}= \lim_{q\to 1}\frac{(1-q)^{n-k}\qraising{q}{q}{n-k}}{\qraising{q^{-n}}{q}{n-k}(1-q)^{n-k}}=(-1)^k\binom{n}{k}^{-1}.
    $$
\end{remark}

It turns out that the $q$-multiplicative convolution can be computed in terms of $q$-differential operators defined in \eqref{def:qderivative}, which can be taken as an alternative definition. This observation, to the best of our knowledge, is new. 

With a polynomial $p\in \P_n$ we can associate its \textbf{$(x\D_q)$-dual} polynomial $\check{p}\in \P_n$, such that
$$
\check{p}(x\D_q) [E_{n,q}(x)]=p(x).
$$
Indeed, $\check{p}$ is the unique polynomial of degree $n$ such that
\begin{equation}
    \check{p}(1-q^k)=\frac{\qraising{q}{q}{k}}{\qraising{q^{-n}}{q}{k}}e_k(p), \qquad \text{for } k=0,1,\dots,n.
\end{equation}

Following the ideas from \cite[Subsection 3.2.4]{mirabelli2021hermitian}, we obtain the following alternative definition of $q$-multiplicative convolution.

\begin{proposition}
 Given polynomials $p, r\in \P_n$, let $\check{p}$ and $\check{r}$ be their $(x\D_q)$-duals. 
Then  
$$
\left( p\boxtimes_{n,q}r\right)(x)=\check{p}(x\D_q)\check{r}(x\D_q)\left[ E_{n,q}(x)\right].
$$
\end{proposition}

\medskip

As it was observed by \cite{lamprecht2016suffridge}, the $q$-multiplicative convolution preserves the real-rootedness of the polynomials under some additional assumptions. We reformulate them here, taking into account the rescaling by $-q^{-n}$ in  Definition~\ref{def:qmultiplication} and the identity \eqref{identityMult2}. Hence (see \cite[Theorem 3]{lamprecht2016suffridge}):
\begin{proposition}   \label{prop:realrootedness}
    Let $p, r \in \P_n$. Then
    \begin{enumerate}[(i)]
        \item $p\in \P^q_n(\rr),\ r\in \P^q_n(\rr_{> 0}) \ \Rightarrow \ p\boxtimes_{n,q} r\in \P^q_n(\rr)$.
        \item $p,r\in \P^q_n(\rr_{> 0}) \ \Rightarrow \ p\boxtimes_{n,q} r\in \P^q_n(\rr_{> 0})$.
    \end{enumerate}
\end{proposition}
Taking into account that $p(-x)=p\boxtimes_{n,q} E_{n,q}(-x)$ (see \eqref{identityMult3} with $\alpha=-1$), a simple consequence of this proposition is that additionally the following ``rule of signs'' applies:
\begin{itemize}
\item $p,r\in \P_{n}^q(\rr_{< 0}) \ \Rightarrow \ p\boxtimes_{n,q} r\in \P^q_n(\rr_{> 0})$,

\item $p\in \P^q_n(\rr_{<0}),\ r\in \P^q_n(\rr_{> 0}) \ \Rightarrow \ p\boxtimes_{n,q} r\in \P^q_n(\rr_{< 0})$.
\end{itemize}
The $q$-multiplicative convolution also preserves zero interlacing:
\begin{proposition}[Preservation of interlacing, {\cite[Theorems 29 and 30]{lamprecht2016suffridge}}]
\label{lem:preservinginterlacingMult}
Let $p, \widetilde p, r\in \pp_n^*(\rr) $. If $r\in \overline{\P^q_n}(\R_{< 0})$,  $p\in \overline{\P_n^q}(\R)$, $\widetilde p\in \P_n(\R)$ and $p \preccurlyeq \widetilde{p}$, then 
    $$ \widetilde p\boxtimes_{n,q} r\preccurlyeq  p\boxtimes_{n,q} r.$$
\end{proposition}
\begin{remark}\label{preserveinterlacing.negativezeros}
Noticing that $p \preccurlyeq r$ if and only if $r(-x) \preccurlyeq p(-x)$ and that $p(x) \boxtimes_{n,q} r(-x)=[p\boxtimes_{n,q} r] (-x)$, we can easily extend Proposition \ref{lem:preservinginterlacingMult} to include polynomials with positive real roots. Namely,  if $p\in\overline{\P_n^q}(\R)$, $\widetilde p\in\P_n(\R)$  and $p \preccurlyeq \widetilde{p}$, then for a polynomial $r$ of degree $n$, 
$$
r\in\overline{\P_n^q}(\R_{> 0}) \quad  \Rightarrow \quad p \boxtimes_{n,q} r \preccurlyeq \widetilde p \boxtimes_{n,q} r.
$$
\end{remark}
Let us now formulate two simple properties of the $q$-multiplicative convolution.
\begin{corollary}\label{cor:lmesh}
 If polynomials $p$, $r$ satisfy  $p\in\overline{\P_n^{q_1}}(\R_{> 0})$ and $r\in\overline{\P_n^{q_2}}(\R_{> 0})$, with $q_1\leq q_2<1$, then $$p\boxtimes_{n,q_2}r\in\overline{\P_n^{q_1}}(\R_{> 0}).$$  
\end{corollary}

\begin{proof}
    Let $p\in\overline{\P_n^{q_1}}(\R_{> 0})$ and $r\in\overline{\P_n^{q_2}}(\R_{> 0})$, then by \eqref{lmeshRplus}, we have 
    $$
    p(x)\preccurlyeq p(q_1x)
    $$
    Taking the $q_2$-multiplicative convolution of $p(x)$ and $p(q_1x)$ with $r(x)$ and using Proposition \ref{lem:preservinginterlacingMult}, we get 
     $$
    p\boxtimes_{n,q_2}r(x)\preccurlyeq p(q_1x)\boxtimes_{n,q_2}r(x)=p\boxtimes_{n,q_2}r(q_1x),
    $$
    which is equivalent to 
$$p\boxtimes_{n,q_2}r\in\overline{\P_n^{q_1}}(\R_{> 0}).$$
\end{proof}

Finally, one can check that $q$-differentiation commutes with $q$-multiplicative convolution up to a constant. This can be seen as a $q$-multiplicative analog of \cite[Remark 3.3]{arizmendi2026s}.
\begin{corollary}
The $q$-multiplicative convolution satisfies that for $p, r\in \pp_n $,
$$\D_q (p\boxtimes_{n,q} r) =\left(1-q^{-n}\right) \D_q (p)\boxtimes_{n-1,q} \D_q (r).
$$
\end{corollary}

\begin{proof}
Notice that for $p\in\P_n$ in the form \eqref{monicP}, using \eqref{derivadaxn} below, we get 
\begin{equation*}
    \D_q p(x)=\sum_{i=0}^{n-1}(1-q^{n-i})e_i(p)x^{n-i-1},
\end{equation*}
so that
\begin{equation*}
    e_i^{(n-1)}(\D_qp)=(1-q^{i})e_{i}^{(n)}(p),
\end{equation*}
where the superscripts $(n-1)$ and $(n)$ are just to emphasize that the dependence of the degree in the coefficient notation \eqref{monicP}. The conclusion now follows from the definition of $\boxtimes_{n,q}$.
\end{proof}

\subsection{\texorpdfstring{Application to $q$-hypergeometric polynomials}{Application to q-hypergeometric polynomials}}\

Our main contribution in Section \ref{sec:3} stems from the observation that the $q$-multiplicative convolution of $q$-hypergeometric polynomials preserves their character, and that the parameters of the resulting $q$-hypergeometric polynomial are obtained by simple concatenation. This is the $q$-analog of the result obtained in \cite{MR4789095}:
\begin{theorem}\label{Theorem:qmultiplicative}
    Let $\bm a_k\in\R^{i_k}$, $\bm b_k\in\R^{j_k}$, $k=1,2$, 
    $$
    p_1(x)=\pQq{i_1+1}{j_1}{q^{-n},\bm a_1}{\bm b_1}{q}{x} \quad \text{ and } \quad p_2(x)=\pQq{i_2+1}{j_2}{q^{-n},\bm a_2}{\bm b_2}{q}{x}.
    $$
   Then
    $$
    p_1(x)\boxtimes_{n,q} p_2(x)=\pQq{i_1+i_2+1}{j_1+j_2+1}{q^{-n},\bm a_1,\bm a_2}{\bm b_1,\bm b_2}{q}{x}
    $$
\end{theorem}
\begin{proof}
   This is a straightforward computation: the $k$-th coefficient of $p_1\boxtimes_{n,q}p_2$ is
    \begin{align*}
        e_k(p_1\boxtimes_{n,q} p_2) & =\frac{\qraising{q}{q}{n-k}}{\qraising{q^{-n}}{q}{n-k}} e_k(p_1) e_k(p_2) \\
        &=(-1)^{(j_1+j_2-i_1-i_2)k}q^{(j_1+j_2-i_1-i_2)\binom{k}{2}}\frac{\qraising{q^{-n}}{q}{k} }{\qraising{q}{q}{k}}\frac{ \qraising{\bm a_1}{q}{k}\qraising{\bm a_2}{q}{k}}{\qraising{\bm b_1}{q}{k}\qraising{\bm b_2}{q}{k}}
    \end{align*}
     which coincides with the expression for the $k$-th coefficient of
    $$
    \pQq{i_1+i_2+1}{j_1+j_2+1}{q^{-n},\bm a_1,\bm a_2}{\bm b_1,\bm b_2}{q}{x}.
    $$
\end{proof}

As a direct consequence of Theorem \ref{Theorem:qmultiplicative} and the results of Section \ref{propertiesqhyper}, we can construct $q$-hypergeometric polynomials with real roots and logarithmic mesh  $\le q$ (see Definition \ref{def:logarithmicmesh}). 

For instance, the $q$-multiplicative convolution of $_2\phi_1$ (with parameters from rows $1$--$6$ of Table~\ref{tab:2Q1}) and $_1\phi_1$ (with parameters from rows $8$--$10$ of the same table) yields information about zero location and lmesh of $_2\phi_2$ polynomials, collected in Table \ref{tab:2Q2}. In particular, row 2 of Table \ref{tab:2Q2} extends the results from \cite[Section 2.1]{kar2020quasi}.

\begin{table}[h]
    \centering
\begin{tabular}{|c|c|c|c|c|c|}
\hline \textbf{Row} & $a$ & $b_1$ & $b_2$ & \textbf{Roots in} & $\lmesh$ \\
\hline \hline 1 & $\R_{<b_1q^{n-1}}$ & $(0,1)$ & $\R_{<1}$ & $\R_{<0}$ & $<q$ \\
\hline 2 & $b_1q^{n-k} $ & $(0,1)$ & $\R_{<1}$ & $\R_{<0}$ & $\leq q$ \\
\hline 3 & $\R_{>q^{1-n}}$ & $\R_{<0}$ & $\R_{<1}$ & $\R_{>0}$ & $<q$ \\
\hline 4 & $(q^{1-n},b_1q^{n+1})$ & $\R_{>q^{-2n+2}}$ &$\R_{<1}$ &  $\R_{<0}$  & $<q$ \\
\hline 5 & $\R_{<0}$ & $\{0\}$ &$\R_{<1}$ & $\R_{<0}$ & $<q$\\
\hline 6 &  $\R_{>q^{1-n}}$  & $\{0\}$ &$\R_{<1}$ & $\R_{>0}$  & $<q$ \\
\hline 7 & $\R_{>q^{1-n}}$ & $\R_{<1}$ & $\R_{<1}$ &{$\R_{> 0}$} & $<q$ \\
\hline
\end{tabular}
\vspace{2mm}
\caption{Real zeros and logarithmic mesh of $\pQq{2}{2}{q^{-n},a}{b_1,b_2}{q}{x}$.} 
    \label{tab:2Q2}
\end{table}   

As the next step, we can apply the $q$-multiplicative convolution of $_2\phi_1$ (with parameters from Table \ref{tab:2Q2}) and $_2\phi_0$ polynomials (with parameters from row 7 of Table \ref{tab:2Q1}) to get information about zero location and lmesh of $_3\phi_2$ polynomials, collected in Table \ref{tab:3Q2}. In particular, row $2$ extends the results from \cite[Section 3]{kar2020quasi}.  

\begin{table}[h]
    \centering
\begin{tabular}{|c|c|c|c|c|c|c|}
\hline \textbf{Row} &$a_1$ & $a_2$ &$b_1$ & $b_2$ & \textbf{Roots in} & $\lmesh$ \\
\hline \hline 1 & $\R_{<b_1q^{n-1}}$ & $(q^{1-n},\infty)$ & $(0,1)$ & $\R_{<1}$ & $\R_{<0}$ & $<q$ \\
\hline 2 & $b_1q^{n-k} $ & $(q^{1-n},\infty)$ &$(0,1)$ & $\R_{<1}$ & $\R_{<0}$ & $\leq q$ \\
\hline 3 & $\R_{>q^{1-n}}$ & $(q^{1-n},\infty)$ &$\R_{<0}$ & $\R_{<1}$ & $\R_{>0}$ & $<q$ \\
\hline 4 & $(q^{1-n},b_1q^{n+1})$ & $(q^{1-n},\infty)$ & $\R_{>q^{-2n+2}}$ &$\R_{<1}$ &  $\R_{<0}$  & $<q$ \\
\hline 5 & $\R_{<0}$ &$(q^{1-n},\infty)$ & $\{0\}$ &$\R_{<1}$ & $\R_{<0}$ & $<q$\\
\hline 6 &  $\R_{>q^{1-n}}$  &$(q^{1-n},\infty)$ & $\{0\}$ &$\R_{<1}$ & $\R_{>0}$  & $<q$ \\
\hline 7 & $\R_{>q^{1-n}}$ &$(q^{1-n},\infty)$ & $\R_{<1}$ & $\R_{<1}$ &{$\R_{> 0}$} & $<q$ \\
\hline
\end{tabular}
\vspace{2mm}
\caption{
Real zeros and logarithmic mesh of $\pQq{3}{2}{q^{-n},\ a_1,\ a_2}{b_1,\ b_2}{q}{x}$. In each appearance in the table above, $k\in \{0,1,\dots,n-1\} $. Recall the polynomial is of degree exactly $n$.} 
    \label{tab:3Q2}
\end{table}   

Table \ref{tab:3Q2} is not exhaustive; it serves rather as an illustration of the information that can be obtained using the technique based on Theorem \ref{Theorem:qmultiplicative}.
 
Clearly, we can go beyond the $_2\phi_2$ and $_3\phi_2$ families to obtain results for polynomials with an arbitrary number of parameters. We can do it iteratively, each time adding a parameter above or below (in the notation \eqref{def:qpolynomial1}) to a real-rooted $q$-hypergeometric polynomial, preserving the location and the upper bound for the logarithmic mesh of the roots. 

\begin{corollary}
    Let $\bm{a}\in\R^{i}$ and $\bm{b}\in\R^j$ be such that
    $$
    \pQq{i+1}{j}{q^{-n},\bm a}{\bm b}{q}{x}\in\P^q_n(\R).
    $$
    Then,
    \begin{enumerate}
        \item for $0\leq \gamma<1$,
        \begin{equation}
        \pQq{i+1}{j+1}{q^{-n},\bm a}{\bm b, \gamma}{q}{x}\in\overline{\P^{q^2}_n}(\R);
        \end{equation}
        \item for $\gamma<0$,
        \begin{equation}
        \pQq{i+1}{j+1}{q^{-n},\bm a}{\bm b, \gamma}{q}{x}\in\P^{q}_n(\R).
        \end{equation}
        \end{enumerate}
         Moreover, for $0<\gamma<1$ and $q^2\leq t<1$, 
        \begin{equation*}
        \pQq{i+1}{j}{q^{-n},\bm a}{\bm b}{q}{x}\in\P^q_n(\R_{>0})\quad \Longrightarrow \quad \pQq{i+1}{j+1}{q^{-n},\bm a}{\bm b, t\gamma}{q}{x}\prec \pQq{i+1}{j+1}{q^{-n},\bm a}{\bm b, \gamma }{q}{x}.
    \end{equation*}
\end{corollary}
On the other hand, convolution with $_2\phi_1$ (see row 1 of Table~\ref{tab:2Q1}) and the interlacing properties \eqref{qinterlacing3}, \eqref{qinterlacing4} lead to the following result:
\begin{corollary}\label{theorem:general}
    Let $\bm{a}\in\R^{i}$, $\bm{b}\in\R^j$, $0<\beta<1$ and $\alpha<\beta q^{n-1}$. Then 
    \begin{equation*}
        \pQq{i+1}{j}{q^{-n},\bm a}{\bm b}{q}{x}\in\P^q_n(\R)\quad \Longrightarrow \quad \pQq{i+1}{j+1}{q^{-n},\bm a,\alpha}{\bm b, \beta}{q}{x}\in\P^{q}_n(\R).
    \end{equation*}
     Moreover, for  $q^2\leq t_1,t_2<1$, if $\alpha>0$,
    \begin{equation*}
        \pQq{i+1}{j}{q^{-n},\bm a}{\bm b}{q}{x}\in\P^q_n(\R_{>0})\quad \Longrightarrow \quad \pQq{i+1}{j+1}{q^{-n},\bm a, t_1\alpha}{\bm b, \beta}{q}{x}\prec \pQq{i+1}{j+1}{q^{-n},\bm a, t_2\alpha}{\bm b, t_2\beta }{q}{x}.
    \end{equation*}
    If $\alpha<0$,
        \begin{equation*}
        \pQq{i+1}{j}{q^{-n},\bm a}{\bm b}{q}{x}\in\P^q_n(\R_{>0})\quad \Longrightarrow \quad \pQq{i+1}{j+1}{q^{-n},\bm a, \alpha}{\bm b, \beta}{q}{x}\prec \pQq{i+1}{j+1}{q^{-n},\bm a, t_2\alpha}{\bm b, \beta }{q}{x}.
    \end{equation*}
\end{corollary}

Another example of results we can obtain by combining the information about the zeros of $_1\phi_1$ and $_2\phi_0$ polynomials from Table \ref{tab:2Q1} is the following:
\begin{corollary}
Let $i, j\geq 0$, and 
$$
0<b_1,\dots,b_j<1, \quad  a_1,\dots,a_{i}>q^{-n+1}.
$$ 
Then, with $\bm a = \left( a_1,\dots,a_{i} \right)$, $\bm b = \left( b_1,\dots,b_{j} \right)$,
$$
\pQq{i+1}{j}{q^{-n},\bm a}{\bm b}{q}{(-1)^j x}\in \P_n(\R_{>0}). 
$$
If, additionally, $q^{2}\leq t < 1$ and $j$ is even, then 
    $$\pQq{i+1}{j}{q^{-n},\ \bm a}{ b_1,\ \dots,\ b_j}{q}{x}\prec\pQq{i+1}{j}{q^{-n},\ \bm a}{ tb_1,\ b_2,\ \dots,\ b_j}{q}{x}, $$
For $j$ odd, the reversed interlacing holds.
\end{corollary}

\begin{example}
Multiple orthogonal polynomials (of type II) are polynomials satisfying orthogonality conditions with respect to $m \geq 1$ positive measures. In particular, given $\beta>-1$ and multi-indices $\bm{n}=(n_1,\cdots,n_m)\in \N^m$ and $\bm \alpha=(\alpha_1,\cdots,\alpha_r)$ such that $\alpha_i>-1$, $\alpha_i-\alpha_j\not\in\Z$, multiple little $q$-Jacobi polynomials of the first kind, denoted by $p_{\bm n}^{(\bm \alpha,\beta)}(x;q)$, are polynomials of degree $N=n_1+\cdots+n_m$ satisfying the orthogonality relations
$$
\int_0^1 p_{\bm n}^{(\bm \alpha,\beta)}(x;q) x^{k+\alpha_j} \frac{(q x ; q)_{\infty}}{\left(q^{\beta+1} x ; q\right)_{\infty}}\, d_q x=0, \quad k=0,1, \ldots, n_j-1, \quad j=1,2, \ldots, m,
$$
where 
$$
\int_0^1 f(x) d_q x=(1-q) \sum_{k=0}^{\infty} q^k f\left(q^k\right).
$$

From \cite[Equation 2.7]{postelmans2005multiple}, we get the following identity, 
   \begin{equation}\label{eq:multiplelittle}
       \pQq{1}{0}{q^{-\beta}}{-}{q}{q^{\beta+1}x}p_{\bm n}^{(\bm \alpha,\beta)}(x;q)\propto\pQq{r+1}{r}{q^{-\beta-N},q^{\bm \alpha +\bm n +1}}{q^{\bm \alpha +1}}{q}{q^{\beta +1}x}.
   \end{equation}
   By \cite{postelmans2005multiple}, it is known that $p_{\bm n}^{(\bm \alpha,\beta)}(x;q)\in \P_N((0,q))$. For $\beta\in\N\cup \{0\}$, the right-hand side of \eqref{eq:multiplelittle} is a $q$-hypergeometric polynomial, which can be decomposed as
   \begin{equation}\label{deco:little}
      \pQq{1}{0}{q^{-\beta}}{-}{q}{q^{\beta+1}x}p_{\bm n}^{(\bm \alpha,\beta)}(x;q)=(r_{\alpha_1}\boxtimes_{n,q}\cdots\boxtimes_{n,q}r_{\alpha_m}) (x),
   \end{equation}
   where

   \begin{equation*}
       r_{\alpha_j}(x):=\pQq{2}{1}{q^{-N},q^{\alpha_j+ n_j +1}}{q^{ \alpha_j +1}}{q}{x}\in\overline{\P^q_N}((0,q)), \qquad \text{for } j=1,\dots,m,
   \end{equation*}
   by row 2 of Table \ref{tab:2Q1}. Proposition \ref{prop:realrootedness} yields now the information on the lmesh:
 \begin{equation}
     p_{\bm n}^{(\bm \alpha,\beta)}(x;q)\in \overline{\P_N^q}((0,q)).
 \end{equation}
 Moreover, by \eqref{qinterlacing6} we have that for $0\leq t<2$,
   \begin{equation}\label{interlacing:r}
       r_{\alpha_j}\preccurlyeq r_{\alpha_j+t}.
   \end{equation}
   Gathering \eqref{deco:little}, \eqref{interlacing:r} and  Proposition \ref{lem:preservinginterlacingMult}, we conclude that
   \begin{equation}
       p_{\bm n}^{(\bm \alpha,\beta)}(x;q)\preccurlyeq p_{\bm n}^{(\bm \alpha+t\bm \theta,\beta)}(x;q),
   \end{equation}
   where $\bm \theta_j$ is the $j$-th row of the $m\times m$ identity matrix. 

\end{example}
 
\section{\texorpdfstring{$q$-Additive Convolution}{q-Additive Convolution}} \label{sec:4}

In this section, we introduce a $q$-analog of the finite free additive convolution.

\subsection{Definition and properties}

\begin{definition}\label{def:qadditive}
Given two polynomials, $p$ and $r$, of degree at most $n$, the 
\textbf{$n$-th $q$-additive finite free convolution} of $p$ and $r$, denoted as $p\boxplus_{n,q} r$, is a polynomial of degree at most $n$, which, following the notational convention \eqref{monicP}, is defined as
$$
[p\boxplus_{n,q} r](x): =\sum_{k=0}^n  e_k(p\boxplus_{n,q} r)\, x^{n-k},
$$
with 
\begin{equation}\label{def:qaddicoeff}
    e_k( p\boxplus_{n,q}r)=(-1)^k\sum_{i+j=k}\frac{\qraising{q}{q}{n-i}\qraising{q}{q}{n-j}}{\qraising{q}{q}{n}\qraising{q}{q}{n-k}}\, e_i(p)e_j(r).
\end{equation}
\end{definition}

The $q$-additive finite free convolution exhibits several convenient properties. It is a commutative and associative linear operator  from $\P_n \times \P_n$ to $\P_n$: if $p,r,g\in \P_n$,  and $\alpha\in\rr$, then 
\begin{equation*}
    (\alpha p+g)\boxplus_{n,q} r= \alpha (p\boxplus_{n,q} r)+g\boxplus_{n,q} r.
\end{equation*}
Moreover, it is compatible with the dilation: for $\alpha\in\rr$,
\begin{equation*}
    p(\alpha x)\boxplus_{n,q} r(\alpha x)=[p\boxplus_{n,q} r](\alpha x).
\end{equation*}
Furthermore, $x^n$ is the identity element for $\boxplus_{n,q}$:
\begin{align}
    \label{identity}
        p \boxplus_{n,q} x^n &= p
        \qquad \text{ for all }p\in\pols.
\end{align}
\begin{remark}[Relation with additive finite free convolution]
Recall that for polynomials $p,r\in\P_n$ in the form \eqref{monicP}, its $n$-th \emph{additive finite free convolution} is 
\begin{equation}\label{coeffAdditiveConv}
    [p\boxplus_n r](x) := \sum_{k=0}^n x^{n-k}(-1)^k \sum_{i+j=k} \frac{(n-i) !(n-j) !}{n !(n-k) !} \, e_i(p) e_j(r).
\end{equation}
Notice that this is the limit of the $q$-additive convolution when $q$ tends to 1:
$$
\lim_{q\to 1}p \boxplus_{n,q} r= p\boxplus_n r.
$$

Indeed, the limit can be checked coefficient-wise using property \eqref{qraising3},
    $$
    \lim_{q\to 1}\frac{\qraising{q}{q}{n-i}\qraising{q}{q}{n-j}}{\qraising{q}{q}{n}\qraising{q}{q}{n-k}}=\frac{(n-i)!(n-j)!}{n!(n-k)!}.
    $$
\end{remark}

It turns out that this convolution also can be computed in terms of $q$-differential operators defined in \eqref{def:qderivative}, which can be taken as an alternative definition.
With a polynomial $p\in \P_n$ we can associate its unique \textbf{$(\D_q)$-dual} polynomial $\widehat{p}\in \P_n$ such that
$$
\widehat{p}(\D_q) [x^n]=p(x).
$$
This correspondence can be made explicit: for 
$$
p(x)=\sum_{j=0}^n  e_j(p)x^{n-j} ,
$$
we have 
\begin{equation}\label{def:qdifferentialoperator}
    \widehat{p}(x)=\sum_{k=0}^n(-1)^k\frac{q^{\binom{k}{2}-nk}}{\qraising{q^{-n}}{q}{k}}e_k(p)x^{k}.
\end{equation}

\begin{proposition}\label{prop:qadditive}
Given polynomials $p, r\in \P_n$, let $\widehat{p}$ and $\widehat{r}$  be their $(\D_q)$-duals.  
Then,
    \begin{align}
       [p\boxplus_{n,q} r](x)&=\widehat{p}(\D_q)\widehat{r}(\D_q)[x^n]\label{prop:qadditive1}.
    \end{align}
\end{proposition}
\begin{proof}
The conclusion follows from a straightforward computation and formula \eqref{def:qdifferentialoperator}:
\begin{align*}
    \widehat{p}(\D_q)\widehat{r}(\D_q)[x^n] &=\widehat{p}(\D_q)r(x)    =\sum_{j=0}^n\sum_{i=0}^n\frac{\qraising{q}{q}{n-i}}{\qraising{q}{q}{n}}e_j(r)e_i(p)\D_q^{i}x^{n-j}\\
    &=\sum_{j=0}^n\sum_{i=0}^n\frac{\qraising{q}{q}{n-i}\qraising{q}{q}{n-j}}{\qraising{q}{q}{n}\qraising{q}{q}{n-j-i}}e_j(r)e_i(p)x^{n-j-i}
    =p\boxplus_{n,q}r(x).
\end{align*} 
\end{proof}

\subsection{Preservation of real roots}\

One of the convenient features of $\boxplus_n$ is that it preserves real roots. Unfortunately, $\boxplus_{n,q}$ does not share this property, in general.
\begin{example}
Let
$$
p(x)=x^2 - 0.80 x + 0.12\in \P^{0.5}_2(\R_{>0}),\text{ and } r(x)=x^2-0.75x+0.09\in \P^{0.5}_2(\R_{>0}).
$$
Nevertheless,
$$
(p\boxplus_{2,0.5}r)(x)=x^2-1.55x+0.61,
$$
which has complex roots. 
\end{example}

However, a simple modification of $\boxplus_{n,q}$ allows us to restore the real roots preservation for polynomials with lmesh  $<q$. Specifically, it turns out that  
$$ p,r\in\P^q_n(\R) \quad \Rightarrow \quad \pQq{1}{1}{q^{-n}}{0}{q}{-x}\boxtimes_{n,q}\left(p\boxplus_{n,q}r\right)\in\P_n(\R).
$$
We prove this fact in two steps. First, we establish a useful identity combining all four finite free convolutions ($\boxtimes_{n}$, $\boxtimes_{n,q}$, $\boxplus_{n}$, and $\boxplus_{n,q}$), which allows us to reduce the claimed result to the known real-roots preservation properties of 
$\boxplus_n$, $\boxtimes_n$, and $\boxtimes_{n,q}$. This is a purely algebraic result of independent interest that does not make any assumption on the location of the roots:
\begin{proposition}\label{prop:qadditive.mod}
For any polynomials $p,  r\in \P_n$,
    \begin{align}
         \pQq{1}{1}{q^{-n}}{0}{q}{-x}\boxtimes_{n,q}\left(p\boxplus_{n,q}r\right)&=[(U_n\boxtimes_{n,q}p)\boxplus_n(U_n\boxtimes_{n,q}r)]\boxtimes_{n}V_n, \label{prop:qadditive2}    
    \end{align}
    where 
    \begin{equation}\label{def:axuliar1}
    U_n(x):=\sum_{k=0}^n\frac{\qraising{q^{-n}}{q}{k}}{k!}x^k,\quad \text{ and }\quad V_n(x):=\sum_{k=0}^n\frac{\raising{-n}{k}}{\qraising{q}{q}{k}}q^{\binom{k}{2}}x^k.
\end{equation}
\end{proposition}
\begin{remark}
It is easy to check that \eqref{prop:qadditive2} is equivalent to the identity
$$
p\boxplus_{n,q}r=[(U_n\boxtimes_{n,q}p)\boxplus_n(U_n\boxtimes_{n,q}r)]\boxtimes_{n}\widehat{V}_n,
$$
where 
$$\widehat{V}_n(x):=\sum_{k=0}^n\frac{\raising{-n}{k}}{\qraising{q}{q}{k}}x^k.$$
\end{remark}
\begin{proof}
We show that the corresponding coefficients of each side of \eqref{prop:qadditive2} match. As agreed, we use the notation for the coefficients as in \eqref{monicP}. 

By the definition of $q$-multiplicative convolution \eqref{def:qmulticoeff},  
\begin{equation*}
    e_j(U_n\boxtimes_{n,q}p)=\frac{\qraising{q}{q}{n-j}}{(n-j)!}e_j(p)\quad \text{and } \quad e_j(U_n\boxtimes_{n,q}r)=\frac{\qraising{q}{q}{n-j}}{(n-j)!}e_j(r).
\end{equation*}
Therefore, by the definition \eqref{coeffAdditiveConv}, for $k\in \{0, 1, \dots,n \}$,
\begin{equation*}
e_k((U_n\boxtimes_{n,q}p)\boxplus_n(U_n\boxtimes_{n,q}r))=\sum_{i+j=k}\frac{\qraising{q}{q}{n-i}\qraising{q}{q}{n-j}}{n!(n-k)!}\, e_j(p)e_i(r).
\end{equation*}
The conclusion follows from the formula for the multiplicative convolution \eqref{coeffMultiplicativeConv}
\begin{align*}
e_k([(U_n\boxtimes_{n,q}p)\boxplus_n(U_n\boxtimes_{n,q}r)]\boxtimes_{n}V_n)&=\sum_{i+j=k}\frac{\qraising{q}{q}{n-i}\qraising{q}{q}{n-j}}{n!\qraising{q}{q}{n-k}}\, e_j(p)e_i(r)\\
    &=\frac{\qraising{q}{q}{n}}{n!}e_k(p\boxplus_{n,q}r).
\end{align*}
\end{proof}

For our second step, we need the preservation properties of real roots for the remaining convolutions appearing in the identity \eqref{prop:qadditive2}. We collected and summarized next the classical results of Szeg\H{o} \cite{szego1922} and Walsh \cite{Walsh1922}, see also \cite{MR4789095}: for $p, g, r \in \P_n$,
\begin{align}
p\prec g,\ r\in \P_n(\rr)&\ \Rightarrow \ p\boxplus_n r\prec g\boxplus_n r \label{eq.add.inter} \\
 p,r\in \P_n(\rr) &\ \Rightarrow \ p\boxplus_n r\in \P_n(\rr), \label{eq.add.real} \\
p,r\in \P_n(\rr_{> 0}) &\ \Rightarrow \ p\boxplus_n r\in \P_n(\rr_{> 0}),\\
p\in \P_n(\rr),\ r\in \P_n(\rr_{> 0}) &\ \Rightarrow \ p\boxtimes_n r\in \P_n(\rr), \label{eq.prod.real} \\ 
p,r\in \P_n(\rr_{> 0}) &\ \Rightarrow \ p\boxtimes_n r\in \P_n(\rr_{> 0}), \\
p\prec g,\ r\in \P_n(\rr_{> 0})&\ \Rightarrow \ p\boxtimes_n r\prec g\boxtimes_n r. \label{eq.prod.inter}
\end{align}
Additionally, we will need the result from \cite{MR2755692} that multiplicative convolution preserves the smallest lmesh: given $q<1$ and polynomials $p$, $r$ such that $p\in\overline{\P_n^{q}}(\R_{> 0})$ and $r\in\overline{\P_n}(\R_{> 0})$, it holds that 
$$
p\boxtimes_{n}r\in\overline{\P_n^{q}}(\R_{> 0}).
$$  

The last ingredient is the following lemma, whose proof is deferred to Appendix~\ref{appendixA}.
\begin{lemma}\label{lem:p.r.real}
For every $n\in\N$, both polynomials \eqref{def:axuliar1} 
are in $ \P^q_n(\R_{>0})$.
\end{lemma}

We are now ready to state and prove the main result of this section, partially announced before:
\begin{theorem}[Preservation of real roots]\label{thm:qadditivereal}
    Let $p,r\in\P_n$. Then 
    \begin{enumerate}
        \item[(i)]$p,r\in\P^q_n(\R) \ \Rightarrow \ \pQq{1}{1}{q^{-n}}{0}{q}{-x}\boxtimes_{n,q}\left(p\boxplus_{n,q}r\right)\in\P_n(\R)$.
    \item[(ii)]$p,r\in\P^q_n(\R_{>0}) \ \Rightarrow \ \pQq{1}{1}{q^{-n}}{0}{q}{-x}\boxtimes_{n,q}\left(p\boxplus_{n,q}r\right)\in\P^q_n(\R_{>0})$.
    \end{enumerate}
\end{theorem}
Notice that in \textit{(i)} we can only claim that the resulting polynomials are real-rooted; no bound on the lmesh is established.
\begin{proof}
We will only show \textit{(i)}, being the proof of the other statement similar.
By assumption, $p,r \in \P^q_n(\R) $, while  Lemma \ref{lem:p.r.real} asserts that $U_n,V_n \in \P^q_n(\R_{>0})$. From Proposition \ref{prop:realrootedness} we get that
$$
U_n\boxtimes_{n,q}p \in \P^q_n(\R) \qquad \text{and} \qquad U_n\boxtimes_{n,q}r\in \P^q_n(\R).
$$
Property \eqref{eq.add.real} then yields that
$$[(U_n\boxtimes_{n,q}p)\boxplus_n(U_n\boxtimes_{n,q}r)]\in \P_n(\R).$$ 
Finally, since $V_n \in \P^q_n(\R_{>0})$,  \eqref{eq.prod.real} implies that
 $$
[(U_n\boxtimes_{n,q}p)\boxplus_n(U_n\boxtimes_{n,q}r)]\boxtimes_{n}V_n \in  \P^q_n(\R).
 $$
 The conclusion follows from the identity \eqref{prop:qadditive2}.
\end{proof}

The same ideas yield an interlacing-preservation result for the corrected $q$-additive convolution.
\begin{theorem}[Preservation of interlacing]\label{thm:qadditiveinterlacing}
 Let $g\in\P_n(\R)$ and  $p,r\in\P_n^q(\R)$  such that $p\prec g$. Then
 \begin{equation}
      \pQq{1}{1}{q^{-n}}{0}{q}{-x}\boxtimes_{n,q}\left(p\boxplus_{n,q}r\right)\prec  \pQq{1}{1}{q^{-n}}{0}{q}{-x}\boxtimes_{n,q}\left(g\boxplus_{n,q}r\right)
 \end{equation}
\end{theorem}
\begin{proof}
    Since $U_n$ belongs to $\P^q_n(\R_{>0})$ (see Lemma \ref{lem:p.r.real}), we get the following facts 
    \begin{align*}
        U_n\boxtimes_{n,q} p \prec U_n\boxtimes_{n,q} g & \quad \text{ (by Proposition \ref{lem:preservinginterlacingMult})},\\
        U_n\boxtimes_{n,q} r\in \P_n^q(\R) & \quad \text{(by Proposition \ref{prop:realrootedness})}.
    \end{align*}
    Using those facts and \eqref{eq.add.inter}, we get 
    $$
    (U_n\boxtimes_{n,q} p)\boxplus_n (U_n\boxtimes_{n,q} r)\prec (U_n\boxtimes_{n,q} g)\boxplus_n (U_n\boxtimes_{n,q} r).
    $$
Finally, since $V_n$ has positive roots, equation \eqref{eq.prod.inter} yields
    $$
    [(U_n\boxtimes_{n,q} p)\boxplus_n (U_n\boxtimes_{n,q} r)]\boxtimes_nV_n\prec [ (U_n\boxplus_{n,q} g)\boxplus_n (U_n\boxtimes_{n,q} r)]\boxtimes_nV_n.
    $$
 The conclusion follows from the identity \eqref{prop:qadditive2}, Proposition \ref{lem:preservinginterlacingMult}, and line 9 of Table \ref{tab:2Q1}.
\end{proof}

To finish this section, let us 
mention another interesting case in which $\boxplus_{n,q}$ preserves the real roots.  The result follows from inductively using the fact that for $\alpha>0$ the operator $\D_q-\alpha$ preserves the $q$-lmesh. Namely, if $p\in\P^q_n(\R)$, then $(\D_q-\alpha)p\in\P^q_n(\R)$.

\begin{theorem}
If $r(x)= \widehat{r}(\D_q) [x^n]$, where $\widehat r\in\P_n(\R_{>0})$, and $p\in\P^q_n(\R_{>0})$ then 
$p\boxplus_{n,q}r\in\P_n^q(\R_{>0})$.
\end{theorem}

\begin{proof}
This is just a reformulation of \cite[Corollary 8.69 part 6]{fisk2008polynomials} after using the interpretation of $\boxplus_{n,q}$ in terms of $\D_q$ given in Proposition \ref{prop:qadditive}.
\end{proof}

\section{\texorpdfstring{$q$-additive convolution of $q$-hypergeometric polynomials}{q-additive convolution of q-hypergeometric polynomials}} \label{sec:5}

In this section, we illustrate the application of the $q$-additive convolution to $q$-hypergeometric polynomials. The main tool will be the expression for $\boxplus_{n,q}$ in terms of $q$-differential operators obtained in Proposition \ref{prop:qadditive}. We provide a general method to translate product identities of $q$-hypergeometric functions into convolutional identities for $q$-hypergeometric polynomials. This is a $q$-analog of the results from \cite{MR4789095} in the hypergeometric setting.

Our approach relies on the notion of duality, discussed above, i.e., formulas of the form 
$$f(\D_q)[x^n]=p(x),$$
where $\D_q$ was defined in \eqref{def:qderivative}, $f$ is a $q$-hypergeometric function, and $p$ is a $q$-hypergeometric polynomial. A basic example of such a formula is
\begin{equation}\label{eq:intuition.phi}
  \pQq{r}{s}{\bm a}{\bm b}{q}{\D_q}[x^n]\propto \pQq{r'}{s'}{q^{-n},\bm b'}{\bm a'}{q}{x}.  
\end{equation}
We will explain how $\bm a$, $\bm b$ are explicitly related to $\bm a'$, $\bm b'$. We further provide variations of this formula in which the right-hand side is evaluated at $cx$ or $x^2$ (instead of $x$), or in which we use $\D_q^2$ or $\D_{q^2}$ on the left-hand side. 

Once we have explicitly computed such formulas, the procedure is as follows: starting with a product identity of $q$-hypergeometric functions,
$$f_1(x)f_2(x)=f_3(x),
$$
we formally evaluate it at $\D_q$ and use Proposition \ref{prop:qadditive}, to arrive at identities of the form
$$
p_1\boxplus_{n,q}p_2=p_3,
$$
where $p_i(x):=f_i(\D_q)[x^n]$ for $i=1,2,3$.

We will study three relevant cases and provide examples of how to use this method for specific polynomials. 

\subsection{\texorpdfstring{From $q$-hypergeometric differential operators to $q$-hypergeometric polynomials}{From q-hypergeometric differential operators to q-hypergeometric polynomials}}\

We begin by providing precise formulas of the form \eqref{eq:intuition.phi} connecting $q$-hypergeometric polynomials and $q$-hypergeometric functions in $\D_q$. The proof of these formulas is rather technical and is presented in Appendix \ref{appendixB}.
\begin{lemma}
\label{lem:qdiff.to.hypergeometric}
Let $\bm a=(a_1,\dots,a_r)\in \R^r$ and $\bm b =(b_1,\dots,b_s)\in \R^s$ be tuples of non-zero numbers satisfying restrictions~\eqref{constrainB}. For $c\neq 0$, define the corresponding constant
    \begin{equation}\label{eq:constants}
         c':= \frac{ b_1b_2\dots b_s}{ca_1a_2\dots a_r}.
    \end{equation} 
Then the following identities hold:
\begin{align}
\pQq{r}{s}{\bm a}{\bm b}{q}{c\D_q}[x^n] 
&\propto \pQq{r+s+1}{r+s}{q^{-n},\bm b^{-1} q^{1-n},\bm 0_{r}}{\bm a^{-1} q^{1-n}, \bm 0_{s}}{q}{ c'qx},
\label{eq:qdifferentialoperator1} \\
\pQq{r}{s}{\bm a }{\bm b}{q}{c\D_{q^2}}[x^n] 
& \propto\pQq{r+s+2}{r+s+1}{q^{-n},\bm b^{-1} q^{1-n},\bm 0_{r+1}}{\bm a^{-1} q^{1-n},-q, \bm 0_{s}}{q}{ c'q x}, \qquad \text{and} 
\label{eq:qdifferentialoperator2} \\
\pQq{r}{s}{\bm a }{\bm b}{q^2}{c\D_{q}^2}[x^{2n}]
&\propto\pQq{r+s+2}{r+s+1}{q^{-2n},\bm b^{-1} q^{2-2n},\bm 0_{r+1}}{\bm a^{-1} q^{2-2n},q, \bm 0_{s}}{q^2}{ c'q^2x^2},
\label{eq:qdifferentialoperator3}
\end{align}
where $\bm 0_r$ and $\bm 0_s$ stand for tuples of $r$ and $s$ zeros, respectively. 
\end{lemma}

\begin{remark}
The formulas above may simplify due to cancellations of the zero tuples in the upper and lower sets of parameters. For instance, in \eqref{eq:qdifferentialoperator1}, if $r=s$ then both zero tuples cancel each other. Otherwise, a single zero tuple of length $|r-s|$ will appear in the upper (if $r>s$) or lower (if $r<s$) set of parameters. A similar analysis is valid for \eqref{eq:qdifferentialoperator2} and \eqref{eq:qdifferentialoperator3}.
 
Furthermore, the restriction that $a_1,\dots, a_r,b_1,\dots, b_s\neq 0$ can be weakened. A more general version of the statement appears as Lemma \ref{lem:qdiff.to.hypergeometric.gen}. It can be derived as a limit case of Lemma \ref{lem:qdiff.to.hypergeometric}, by letting a parameter tend to zero. For instance, if $a_i\to 0$, then $a_i^{-1} q^{1-n}\to\infty$, and the corresponding parameter on the right-hand side of \eqref{eq:qdifferentialoperator1} disappears at the cost of scaling the polynomial by $q^{1-n}$. A similar reasoning works for \eqref{eq:qdifferentialoperator2} and \eqref{eq:qdifferentialoperator3}.
\end{remark}

In the rest of the section, we derive $q$-additive convolution formulas from the identities  \eqref{eq:qdifferentialoperator1}--\eqref{eq:qdifferentialoperator3}. Instead of being exhaustive, we focus on the three most common types of product formulas.

\subsection{\texorpdfstring{Formulas with $\boxplus_{n,q}$}{Formulas with boxplus n,q}}

We begin with the simplest product formula involving only standard $q$-hypergeometric polynomials.

\begin{theorem}
\label{thm:additive.1}
Let $\bm a_k\in\R^{r_k}$, $\bm b_k\in\R^{s_k}$, $k=1,2,3$, with non-zero entries satisfying restrictions~\eqref{constrainB}. Given nonzero constants $c_1,c_2,c_3$, define the corresponding constants $c'_1,c'_2,c'_3$ in terms of the parameters as in \eqref{eq:constants}.

If
     \begin{equation}\label{product1}
        \pQq{r_1}{s_1}{\bm a_1 }{\bm b_1}{q}{c_1x}\pQq{r_2}{s_2}{\bm a_2 }{\bm b_2}{q}{c_2x}=\pQq{r_3}{s_3}{\bm a_3 }{\bm b_3}{q}{c_3x},
    \end{equation}
then the $q$-additive convolution of
 $$
    p_1(x):=\pQq{r_1+s_1+1}{r_1+s_1}{q^{-n}\, ,\,\bm b_1^{-1} q^{1-n}\, ,\,\bm 0_{r_1}}{\bm a_1^{-1} q^{1-n}\, ,\, \bm 0_{s_1}}{q}{ c_1'qx} 
    $$
and
    $$
     p_2(x):= \pQq{r_2+s_2+1}{r_2+s_2}{q^{-n}\, ,\, \bm b_2^{-1} q^{1-n}\, ,\, \bm 0_{r_2}}{\bm a_2^{-1} q^{1-n}\, ,\, \bm 0_{s_2}}{q}{ c_2'qx}
    $$
satisfies
$$
p_1\boxplus_{n,q}p_2\propto \pQq{r_3+s_3+1}{r_3+s_3}{q^{-n}\, ,\,\bm b_3^{-1} q^{1-n}\, ,\,\bm 0_{r_3}}{\bm a_3^{-1} q^{1-n}\, ,\, \bm 0_{s_3}}{q}{c_3'qx}.
$$
\end{theorem}
\begin{proof}
 We formally replace the variable $x$ by $\D_q$ in \eqref{product1} and apply the resulting $q$-differential operator to $x^n$: 
   $$
   \pQq{r_1}{s_1}{\bm a_1}{\bm b_1}{q}{c_1\D_q}\pQq{r_2}{s_2}{\bm a_2 }{\bm b_2}{q}{c_2\D_q}[x^n]=\pQq{r_3}{s_3}{\bm a_3}{\bm b_3}{q}{c_3\D_q}[x^n].
   $$
The conclusion is a straightforward consequence of \eqref{prop:qadditive1} and \eqref{eq:qdifferentialoperator1}.
\end{proof}

We now compile some known product formulas of the form \eqref{product1} and use them to obtain examples of $q$-additive convolutions.
\begin{example}
The $q$-binomial theorem (see for instance \cite[Equation (1.11.1)]{koekoek2010hypergeometric}) can be written as the product
\begin{equation*}
    \pQq{0}{0}{\cdot\,}{\cdot\,}{q}{ax}=\pQq{0}{0}{\cdot\,}{\cdot\,}{q}{x}\pQq{1}{0}{a}{\cdot\,}{q}{x}.
\end{equation*}
By Theorem \ref{thm:additive.1}, we get
\begin{equation}
    \pQq{1}{0}{q^{-n}}{\cdot}{q}{a^{-1}qx}\propto\pQq{1}{0}{q^{-n}}{\cdot}{q}{qx}\boxplus_{n,q}\pQq{2}{1}{q^{-n},0}{q^{1-n}a^{-1}}{q}{a^{-1}qx}.
\end{equation}
Further simplification can be achieved by changing of variables $a\mapsto  a^{-1}$ and $qx \mapsto x$:
\begin{equation}
    \pQq{1}{0}{q^{-n}}{\cdot}{q}{ax}\propto\pQq{1}{0}{q^{-n}}{\cdot}{q}{x}\boxplus_{n,q}\pQq{2}{1}{q^{-n},0}{aq^{1-n}}{q}{ax}.
\end{equation}
\end{example}

\begin{example}
Heine's transformation formula, which is the $q$-analog of Euler’s transformation formula, see \cite[Equation (1.13.3)]{koekoek2010hypergeometric}, can be also written as the product
\begin{equation*}
    \pQq{2}{1}{a_1,a_2}{b}{q}{x}=\pQq{1}{0}{a_1a_2 b^{-1}}{\cdot}{q}{x}\pQq{2}{1}{a_1^{-1}b,a_2^{-1}b}{b}{q}{a_1a_2b^{-1}x}.
\end{equation*}
Theorem \ref{thm:additive.1} then asserts that
\begin{align*}
 &\pQq{2}{1}{q^{-n},0}{a_1^{-1}a_2^{-1}bq^{1-n}}{q}{a_1^{-1}a_2^{-1}bq x}\boxplus_{n,q}\pQq{3}{2}{q^{-n},b^{-1}q^{1-n},0}{a_1b^{-1}q^{1-n},a_2b^{-1}q^{1-n}}{q}{qx} \\
  & \propto\pQq{3}{2}{q^{-n},b^{-1}q^{1-n},0}{a_1^{-1}q^{1-n},a_2^{-1}q^{1-n}}{q}{a_1^{-1}a_2^{-1}bqx} .
\end{align*}
Relabeling as $x:=a_1^{-1}a_2^{-1}bq^2x$, $b_1:=a_1^{-1}q^{1-n}$, $b_2:=a_2^{-1}q^{1-n}$ and $a:=b^{-1}q^{1-n}$, we 
get
\begin{equation*}
 \pQq{2}{1}{q^{-n},0}{a^{-1}b_1b_2}{q}{x}\boxplus_{n,q}\pQq{3}{2}{q^{-n},a,0}{ab_1^{-1}q^{1-n},ab_2^{-1}q^{1-n}}{q}{ab_1^{-1}b_2^{-1}q^{1-n}x}  \propto   \pQq{3}{2}{q^{-n},a,0}{b_1,b_2}{q}{x}.
\end{equation*}

One can derive simpler formulas by taking limits. For instance, if we let $b_2\to 0$ and denote $b:=b_1$ to ease notation, we obtain
\begin{equation}\label{eq:Heine.limit}
\pQq{2}{1}{q^{-n}}{\cdot}{q}{x}\boxplus_{n,q}\pQq{3}{2}{q^{-n},a,0}{ab^{-1}q^{1-n}}{q}{b^{-1}x} 
\propto \pQq{3}{2}{q^{-n},a}{b}{q}{x}.
\end{equation}
Notice that one can alternatively obtain \eqref{eq:Heine.limit} by first taking the limit in Heine’s transformation formula, see \cite[Equation (1.13.7)]{koekoek2010hypergeometric}, and then using a generalized version of Theorem \ref{thm:additive.1} that allows parameters to be equal to 0. Such generalization follows from using Lemma \ref{lem:qdiff.to.hypergeometric.gen}.

In \eqref{eq:Heine.limit}, we can further take the limit $b\to 0$ to obtain
\begin{equation*}
\pQq{2}{1}{q^{-n}}{\cdot}{q}{x}\boxplus_{n,q}\pQq{3}{2}{q^{-n},a,0}{\cdot}{q}{a^{-1}q^{n-1}x}  \propto   \pQq{3}{2}{q^{-n},a}{0}{q}{x}.
\end{equation*}
Again, one can obtain this formula directly from  \cite[Equation (1.13.9)]{koekoek2010hypergeometric} using Lemma \ref{lem:qdiff.to.hypergeometric.gen}.
\end{example}

\subsection{\texorpdfstring{Formulas with $\boxplus_{2n,q}$}{Formulas with boxplus 2n,q}}\

Another common type of product formulas has standard $q$-hypergeometric functions as factors, but the product is a $q^2$-hyper\-geo\-met\-ric function evaluated at $x^2$. They also lead to convolution identities of $q$-hypergeometric polynomials:
\begin{theorem}
\label{thm:additive.2}
Let $\bm a_k\in\R^{r_k}$, $\bm b_k\in\R^{s_k}$, $k=1,2,3$, with non-zero entries satisfying restrictions~\eqref{constrainB}. Given nonzero constants $c_1,c_2,c_3$, define the corresponding constants $c'_1,c'_2,c'_3$ in terms of the parameters as in \eqref{eq:constants}.

If
     \begin{equation}\label{product2}
        \pQq{r_1}{s_1}{\bm a_1}{\bm b_1}{q}{c_1x}\pQq{r_2}{s_2}{\bm a_2}{\bm b_2}{q}{c_2x}=\pQq{r_3}{s_3}{\bm a_3}{\bm b_3}{q^2}{c_3x^2},
    \end{equation}
then the $q$-additive convolution of 
 $$
    p_1(x):=\pQq{r_1+s_1+1}{r_1+s_1}{q^{-2n},\bm b_1^{-1} q^{1-2n},\bm 0_{r_1}}{\bm a_1^{-1} q^{1-2n}, \bm 0_{s_1}}{q}{ c'_1qx} 
    $$
and
    $$
     p_2(x):= \pQq{r_2+s_2+1}{r_2+s_2}{q^{-2n},\bm b_2^{-1} q^{1-2n},\bm 0_{r_2}}{\bm a_2^{-1} q^{1-2n}, \bm 0_{s_2}}{q}{ c'_2qx}
    $$
satisfies
  $$
p_1\boxplus_{2n,q}p_2 \propto \pQq{r_3+s_3+2}{r_3+s_3+1}{q^{-2n},\bm b_3^{-1} q^{2-2n},\bm 0_{r_3+1}}{\bm a_3^{-1} q^{2-2n},q, \bm 0_{s_3}}{q^2}{ c'_3 q^2 x^2}.
$$
\end{theorem}
\begin{proof}
 We evaluate \eqref{product2} at $\D_q$ and apply the resulting $q$-differential operator to $x^{2n}$ to obtain
   $$
   \pQq{r_1}{s_1}{\bm a_1}{\bm b_1}{q}{c_1\D_q}\pQq{r_2}{s_2}{\bm a_2}{\bm b_2}{q}{c_2\D_q}[x^{2n}]=\pQq{r_3}{s_3}{\bm a_3}{\bm b_3}{q^2}{c_3\D_q^2}[x^{2n}].
   $$
The conclusion follows from \eqref{prop:qadditive1} and Lemma \ref{lem:qdiff.to.hypergeometric}. Specifically, we use \eqref{eq:qdifferentialoperator1} on the left-hand side, and \eqref{eq:qdifferentialoperator3} on the right-hand side of the equality above.
\end{proof}

We now apply this result to the case where $c_1=1$ and $c_2=-1$. Notice that these examples can be seen as a $ q$-analog of symmetrization of polynomials as studied in \cite[Sections 3.2 and 4.2]{campbell2025even}. This notion can be defined for any polynomial:
\begin{notation}
\label{not:qsym}
The \emph{$q$-symmetrization} of $p \in \P_n$ is the polynomial
    \begin{equation}\label{symmdefinition}
        \qsym{p}:=p\boxplus_{n,q} \left(\dil{-1} p \right) \text{,}
    \end{equation}
    where $(\dil{-1} p)(x):=(-1)^np(-x)$.
\end{notation}

From \eqref{def:qaddicoeff} it is readily seen that the polynomials $\qsym{p}$ are even (their odd coefficients vanish). Further properties of symmetrization are straightforward (see also \cite[Sections 3.2]{campbell2025even}): for $p,r \in \P_n$,
   \begin{enumerate}
        \item $\qsym{p} = \qsym{\dil{-1} p}$;
        \item If $p$ is an even polynomial, then $\qsym{p} = p \boxplus_{n,q} p$;
        \item $\qsym{p \boxplus_{n,q} r} = \qsym{p} \boxplus_{n,q} \qsym{r}$;
        \item $\qsym{\dil{\alpha} p} = \dil{\alpha} \qsym{p}$.
    \end{enumerate}

Next, we use Theorem \ref{thm:additive.2} to compute $\qsym{p}$ for some $q$-hypergeometric polynomials $p$; these results can be seen as $q$-analogs of results in \cite[Table 7]{campbell2025even}.
\begin{example}
The $q$-extension of Bailey's formula reads as \cite[Equation (4.9)]{srivastava1986q}  
    \begin{equation*}
        \pQq{4}{3}{ab,-ab,abq,-abq}{a^2q,b^2q,a^2b^2}{q^2}{x^2}=\pQq{2}{1}{a,-a}{a^2}{q}{x}\pQq{2}{1}{b,-b}{b^2}{q}{-x}.
    \end{equation*}
Applying Theorem \ref{thm:additive.2}, we get a formula for the $q$-symmetrization of a ${}_3\phi_2$ polynomial:
\begin{multline*}
 \pQq{3}{2}{q^{-2n},a^{-2}q^{1-2n},0}{a^{-1}q^{1-2n},-a^{-1}q^{1-2n}}{q}{qx}\boxplus_{2n,q} \pQq{3}{2}{q^{-2n},b^{-2}q^{1-2n},0}{b^{-1}q^{1-2n},-b^{-1}q^{1-2n}}{q}{-qx} \\
   \propto \pQq{5}{4}
    {q^{-2n},a^{-2}q^{1-2n},b^{-2}q^{1-2n},a^{-2}b^{-2}q^{2-2n},0,0}{a^{-1}b^{-1}q^{2-2n},-a^{-1}b^{-1}q^{2-2n},a^{-1}b^{-1}q^{1-2n},-a^{-1}b^{-1}q^{1-2n},q}
    {q^2}{q^2x^2}.
\end{multline*}
\end{example}

\begin{example}
Formula \cite[Equation (3.13)]{srivastava1986some},
    \begin{equation*}
        \pQq{4}{3}{a^2,b^2,ab,abq}{a^2b^2,-ab,-abq}{q^2}{x^2}=\pQq{2}{1}{a,b}{-ab}{q}{x}\pQq{2}{1}{a,b}{-ab}{q}{-x}
    \end{equation*}
and Theorem \ref{thm:additive.2} lead to
\begin{multline*}
\pQq{3}{2}{q^{-2n},-a^{-1}b^{-1}q^{1-2n},0}{a^{-1}q^{1-2n},b^{-1}q^{1-2n}}{q}{qx}\boxplus_{2n,q} \pQq{3}{2}{q^{-2n},-a^{-1}b^{-1}q^{1-2n},0}{a^{-1}q^{1-2n},b^{-1}q^{1-2n}}{q}{-qx}\\ 
  \propto  \pQq{5}{4}
    {q^{-2n},a^{-2}b^{-2}q^{2-2n},-a^{-1}b^{-1}q^{2-2n},-a^{-1}b^{-1}q^{1-2n},0,0}{a^{-2}q^{2-2n},b^{-2}q^{2-2n},a^{-1}b^{-1}q^{2-2n},a^{-1}b^{-1}q^{1-2n},q}
    {q^2}{q^2x^2}.
\end{multline*}

\end{example}

\begin{example}
An interesting $q$-extension of Ramanujan's formula, found in \cite[Equation (3.3)]{srivastava1986some}, is  
    \begin{equation*}
      \pQq{4}{3}{a^2,-b,-bq,a^{-2}b^2}{b^2,b,bq}{q^2}{x^2}  =\pQq{2}{1}{a,-a^{-1}b}{b}{q}{x}\pQq{2}{1}{a,-a^{-1}b}{b}{q}{-x}.
    \end{equation*}
 Applying Theorem \ref{thm:additive.2}, we get 
    \begin{multline}\label{eq:exm.Ramanujan} 
    \pQq{3}{2}{q^{-2n},b^{-1}q^{1-2n},0}{a^{-1}q^{1-2n},-ab^{-1}q^{1-2n}}{q}{qx}
        \boxplus_{2n,q}
        \pQq{3}{2}{q^{-2n},b^{-1}q^{1-2n},0}{a^{-1}q^{1-2n},-ab^{-1}q^{1-2n}}{q}{-qx} \\
      \propto  \pQq{5}{4}{q^{-2n},b^{-2}q^{2-2n},b^{-1}q^{2-2n},b^{-1}q^{1-2n},0,0}{a^{-2}q^{2-2n},-b^{-1}q^{2-2n},-b^{-1}q^{1-2n},a^2b^{-2}q^{2-2n},q}{q^2}{q^2x} .
    \end{multline}

We can simplify these formulas by taking limits on the parameters. With $b\to \infty$ and  $a:= a^{-1}q^{1-2n}$, $x:=qx$, we obtain
 \begin{equation*}
        \pQq{2}{1}{q^{-2n},0}{a}{q}{x}
        \boxplus_{2n,q}
        \pQq{2}{1}{q^{-2n},0}{a}{q}{-x} 
\propto  \pQq{3}{2}{q^{-2n},0,0}{a^{2}q^{2n},q}{q^2}{x^2}.
    \end{equation*}
Notice that one can alternatively obtain this formula by considering \cite[Equation (3.11)]{srivastava1986some}, and using Lemma \ref{lem:qdiff.to.hypergeometric.gen} to obtain a strengthened version of Theorem \ref{thm:additive.2}, allowing parameters to be equal to 0.

 If instead, in \eqref{eq:exm.Ramanujan} we let $a\to 0$, then 
    \begin{multline*}
        \pQq{5}{4}{q^{-2n},b^{-2}q^{2-2n},b^{-1}q^{2-2n},b^{-1}q^{1-2n},0}{-b^{-1}q^{2-2n},-b^{-1}q^{1-2n},q}{q^2}{q^{2n}x}\propto\\
        \pQq{3}{2}{q^{-2n},b^{-1}q^{1-2n},0}{\cdot}{q}{q^{2n}x}
        \boxplus_{2n,q}
        \pQq{3}{2}{q^{-2n},b^{-1}q^{1-2n}}{\cdot}{q}{-q^{2n}x}
    \end{multline*}
Alternatively, this formula follows from \cite[Equation (3.12)]{srivastava1986some}, and Lemma \ref{lem:qdiff.to.hypergeometric.gen}.
\end{example}

\subsection{\texorpdfstring{Formulas with $\boxplus_{n,q^2}$}{Formulas with boxplus n, q squared}}\

Finally, we consider the case where the factors in the formula are $q^2$-hypergeometric functions. Now, the proof requires evaluating the product formula in $\D_{q^2}$ and using the identities with $q^2$ instead of $q$.
\begin{theorem}
\label{thm:additive.3}
Let $\bm a_k\in\R^{r_k}$, $\bm b_k\in\R^{s_k}$, $k=1,2,3$, with non-zero entries satisfying restrictions~\eqref{constrainB}. Given nonzero constants $c_1,c_2,c_3$, define the corresponding constants $c'_1,c'_2,c'_3$ in terms of the parameters as in \eqref{eq:constants}.

If
     \begin{equation} \label{product3}
        \pQq{r_1}{s_1}{\bm a_1 }{\bm b_1}{q^2}{c_1x}\pQq{r_2}{s_2}{\bm a_2 }{\bm b_2}{q^2}{c_2x}=\pQq{r_3}{s_3}{\bm a_3 }{\bm b_3}{q}{c_3x},
    \end{equation}
then the $q$-additive convolution of 
 $$
    p_1(x):=\pQq{r_1+s_1+1}{r_1+s_1}{q^{-2n},\bm b_1^{-1} q^{2-2n},\bm 0_{r_1}}{\bm a_1^{-1} q^{2-2n}, \bm 0_{s_1}}{q^2}{c'_1q^2x}
    $$
and
    $$
     p_2(x):= \pQq{r_2+s_2+1}{r_2+s_2}{q^{-2n},\bm b_2^{-1} q^{2-2n},\bm 0_{r_2}}{\bm a_2^{-1} q^{2-2n}, \bm 0_{s_2}}{q^2}{c'_2q^2x}
    $$
satisfies
  $$
p\boxplus_{n,q^2}r \propto \pQq{r_3+s_3+2}{r_3+s_3+1}{q^{-n},\bm b_3^{-1} q^{1-n},\bm 0_{r_3+1}}{\bm a_3^{-1} q^{1-n},-q, \bm 0_{s_3}}{q}{c'_3qx}.
$$
\end{theorem}
\begin{proof}
 We evaluate \eqref{product3} at $\D_{q^2}$ and apply the resulting operator to $x^{n}$, obtaining
   $$
   \pQq{r_1}{s_1}{\bm a_1 }{\bm b_1}{q^2}{c_1\D_{q^2}}\pQq{r_2}{s_2}{\bm a_2 }{\bm b_2}{q^2}{c_2\D_{q^2}}[x^{n}]=\pQq{r_3}{s_3}{\bm a_3 }{\bm b_3}{q}{c_3\D_{q^2}}[x^{n}].
   $$
The conclusion follows from \eqref{prop:qadditive1} and Lemma \ref{lem:qdiff.to.hypergeometric}. Specifically, on the right-hand side, we use \eqref{eq:qdifferentialoperator2}, while on the left-hand side, we use \eqref{eq:qdifferentialoperator1} with $q^2$ instead of $q$.
\end{proof}

We now provide some examples:
\begin{example}
The $q$-extension of Clausen's identity is \cite[Equation 4.12]{srivastava1986q},
    \begin{equation*}
      \pQq{4}{3}{a, b, \sqrt{ab}, -\sqrt{ab}}{ab, \sqrt{abq},-\sqrt{abq}}{q}{x}  =\pQq{2}{1}{a,b}{abq}{q^2}{x} \pQq{2}{1}{a, b}{abq}{q^2}{qx}.
    \end{equation*}
Using Theorem \ref{thm:additive.3}, and normalizing $x:=q^2x$ we obtain
\begin{multline*}
  \pQq{3}{2}{q^{-2n},a^{-1}b^{-1}q^{1-2n},0}{a^{-1}q^{2-2n},b^{-1}q^{2-2n}}{q^{2}}{qx}\boxplus_{n,q^{2}}\pQq{3}{2}{q^{-2n},a^{-1}b^{-1}q^{1-2n},0}{a^{-1}q^{2-2n},b^{-1}q^{2-2n}}{q^{2}}{x} \\
  \propto   \pQq{5}{4}{q^{-n},a^{-1}b^{-1}q^{1-n},(ab)^{-1/2}q^{1/2-n},-(ab)^{-1/2}q^{1/2-n},0,0}{a^{-1}q^{1-n},b^{-1}q^{1-n},(ab)^{-1/2}q^{1-n},-(ab)^{-1/2}q^{1-n},-q}{q}{x}.  
\end{multline*}
\end{example}

\begin{example}
From \cite[Equation 4.11]{srivastava1986q} we get
    \begin{equation*}
        \pQq{4}{3}{ab,-ab,ab\sqrt{q},-ab\sqrt{q}}{a^2b^2,-a^2,-b^2q}{q}{x}=\pQq{2}{1}{a^2,a^2q}{a^4}{q^2}{x}\pQq{2}{1}{b^2,b^2q}{b^4q^2}{q^2}{qx}.
    \end{equation*}
Using Theorem \ref{thm:additive.3}, and letting $x:=qx$ we get
\begin{multline*}
  \pQq{3}{2}{q^{-2n},a^{-4}q^{2-2n},0}{a^{-2}q^{2-2n},a^{-2}q^{1-2n}}{q^{2}}{x}\boxplus_{n,q^{2}}\pQq{3}{2}{q^{-2n},b^{-4}q^{-2n},0}{b^{-2}q^{2-2n},b^{-2}q^{1-2n}}{q^{2}}{qx} \\
 \propto  \pQq{6}{5}{q^{-n},a^{-2}b^{-2}q^{1-n},-a^{-2}q^{1-n},-b^{-2}q^{-n},0,0}{a^{-1}b^{-1}q^{1-n},-a^{-1}b^{-1}q^{1-n},a^{-1}b^{-1}q^{\frac{1}{2}-n},-a^{-1}b^{-1}q^{\frac{1}{2}-n},-q}{q}{x}  .
\end{multline*}
Sending $a$ or $b$ to $0$ or $\infty$, we get some interesting formulas. Alternatively, one can obtain these cases directly from \cite[Equation (3.13)--(3.16)]{srivastava1986some}, and using Lemma \ref{lem:qdiff.to.hypergeometric.gen} to obtain a generalized version of Theorem \ref{thm:additive.3} that allows parameters to be equal to 0. 

For instance, with $b\to \infty$ we get
\begin{equation*} 
\pQq{3}{2}{q^{-2n},a^{-4}q^{2-2n},0}{a^{-2}q^{2-2n},a^{-2}q^{1-2n}}{q^{2}}{x}\boxplus_{n,q^{2}}\pQq{1}{0}{q^{-2n}}{\cdot}{q^{2}}{qx} 
\propto \pQq{2}{1}{q^{-n},-a^{-2}q^{1-n}}{-q}{q}{x}.
\end{equation*}
If we let $a\to \infty$, then 
\begin{equation*}
  \pQq{1}{0}{q^{-2n}}{\cdot}{q^{2}}{x}\boxplus_{n,q^{2}}\pQq{3}{2}{q^{-2n},b^{-4}q^{-2n},0}{b^{-2}q^{2-2n},b^{-2}q^{1-2n}}{q^{2}}{qx} \propto     \pQq{2}{1}{q^{-n},-b^{-2}q^{-n}}{-q}{q}{x}.
\end{equation*}
\end{example}

\appendix
 
\section{Polynomials with real roots.} \label{appendixA}

Here we will prove both claims of Lemma \ref{lem:p.r.real}: 
\begin{equation*}
    U_n(x):=\sum_{k=0}^n\frac{\qraising{q^{-n}}{q}{k}}{k!}x^k \in \P^q_n(\R_{>0})\quad \text{ and }\quad V_n(x):=\sum_{k=0}^n\frac{\raising{-n}{k}}{\qraising{q}{q}{k}}q^{\binom{k}{2}}x^k \in \P^q_n(\R_{>0}).
\end{equation*}
\begin{proof}[Proof of Lemma \ref{lem:p.r.real}] We prove our assertions by induction on the degree $n$. 

\underline{$U_n\in \P^q_n(\R_{>0})$}: we will use the following straightforward identities: for $n\ge 2$,
\begin{align}
U'_n(x)&=(1-q^{-n})U_{n-1}(x), \label{auxiliarpderivative}\\
(1-q^n)U_n(x)&=x\D_qU_n(x)+(1-q^n)U_{n-1}(x). \label{eq:pnpn-1}
\end{align}

Now, for $n=2$ we notice that
    \begin{equation*}
        U_2(x)=1+(1-q^{-2})x+\tfrac{(1-q^{-2})(1-q^{-1})}{2}\, x^2,
    \end{equation*}
has roots
$$
\lambda_1(U_{2})=q\tfrac{1-q^2-\sqrt{(1-q)^3(q+1)}}{(1-q)(1-q^2)}, \quad  \lambda_2(U_{2})=q\tfrac{1-q^2+\sqrt{(1-q)^3(q+1)}}{(1-q)(1-q^2)},
$$
with $0<\lambda_1(U_{2})<\lambda_2(U_{2})$. Thus, its logarithmic root separation is
$$
\lmesh{U_2}=\frac{\lambda_1(U_{2})}{\lambda_2(U_{2})}=\frac{1-q^2-\sqrt{(1-q)^3(q+1)}}{1-q^2+\sqrt{(1-q)^3(q+1)}}=\frac{1-\sqrt{1-q^2}}{q}<q.
$$
Therefore, $U_2\in\P_2^q(\R_{>0})$. 

For $n>2$, let us denote
$$
\alpha_j:=\lambda_j(U_{n-1}), \quad j=1, \dots, n-1.
$$
By induction hypothesis we assume that $U_{n-1}\in\P^q_{n-1}(\R_{>0})$, which by \eqref{lmeshRplus} implies 
$$
\alpha_1  < q\alpha_2  < \alpha_2  < \dots < \alpha_{n-2}  <  q\alpha_{n-1}  < \alpha_{n-1}.
$$
By \eqref{auxiliarpderivative}, we know that $U'_{n}=(1-q^{-n})U_{n-1}$, so the zeros of $U_{n-1}$ are the critical values of $U_{n}$, and $U_n$ is strictly monotonic on each interval $(\alpha_j , \alpha_{j+1}) $, $j=0,\cdots,n-1$, with the convention of $\lambda_0 =-\infty$ and $\lambda_n =+\infty$. Using the monotonicity of $U_n$, we obtain for
$j\in\{1,\cdots n-2\}$,
\begin{equation}\label{pmonotonicity}
    U_n(q\, \alpha_{j} )<U_n(\alpha_{j} ) \quad \Longrightarrow \quad  U_n(q\, \alpha_{j+1} )>U_n(\alpha_{j+1} ).
\end{equation}
which implies, 
\begin{equation}\label{pchangeofsign}
\sign{ U_n(\alpha_{j} ) -U_n(q\, \alpha_{j}) } =-\sign{ U_n(\alpha_{j+1})-U_n(q\, \alpha_{j+1}) }.
\end{equation}
By the definition of $\D_q$,
$$
\D_qU_n(x)=\frac{U_n(x)-U_n(qx)}{x},
$$
and \eqref{pchangeofsign} shows that $\D_qU_n$ changes of sign in each interval $(\alpha_j,\alpha_{j+1})$; this gives us the location of $n-1$ zeros. If we assume that for some $j\in\{1,\cdots,n-1\}$,  $\D_qU_n(\alpha_j)=0$, then replacing it into the expression for $\D_q$ would give us that $U_n(q\alpha_j)=U_n(\alpha_j)$; this contradicts \eqref{pmonotonicity}. Hence, the remaining zero of $\D_qU_n$ belongs either to $(-\infty,\alpha_1)$ or $(\alpha_{n-1},+\infty)$.

For $U_n$, the leading coefficient is $(-1)^n$ and since its degree is $n$, then $U_n$ decrease in $(-\infty,\alpha_1)$, this implies that 
$$
\D_qU_n(\alpha_1)<0.
$$
Using that the sign of the leading coefficient of $\D_qU_n$ is $(-1)^n$ and that its degree is $n-1$, we get that $\D_qU_n$ is increasing and negative on $(-\infty,\lambda_1(\D_qU_n))$. Then $\alpha_1<\lambda_1(\D_qU_n)$, hence $\D_qU_n$ does not have zeros on the interval $(-\infty,\alpha_1)$. Therefore, that last zero must belong in $(\alpha_{n-1},+\infty)$. Note that the location of each zero implies $U_{n-1}\prec\D_q(U_n)$,
or equivalently
\begin{equation}\label{eq:interlacinp}
   x\D_qU_n(x) \prec U_{n-1}(x).
\end{equation}
Note that \eqref{eq:pnpn-1} satisfies the conditions of Lemma \ref{Jordaaninterlacing} and using \eqref{eq:interlacinp}, we get that 
$$
U_n(x)\prec x\D_qU_n(x) \quad\text{or}\quad x\D_qU_n(x)\prec U_n(x).
$$
Since $U_n(0)=1$, $U_n$ decreases on $(-\infty,\alpha_1)$, and $\alpha_1$ is the smallest critical value of $U_n$, we conclude that all its zeros are strictly positive, therefore 
$$
x\D_qU_n(x)\prec U_n(x)\quad \Longrightarrow \quad U_n\prec \D_q(U_n).
$$
Then Proposition \ref{prop:lmeshinterlaing} yields $U_n\in \P^q_n(\R_{>0})$.

\medskip

\underline{$V_n\in \P^q_n(\R_{>0})$}: we will use the following straightforward identities: for $n\ge 2$, 
\begin{align}
nV_n (x) & =- W_n(q^{-1}x) +xV_n'(x) ,\quad W_n (x):= \D_qV_n(x), \label{eq:rn} \\
         -nV_{n-1}(x)&=W_n(q^{-1}x) . \label{auxiliarpderivative2}
\end{align}

Again, for $n=2$, the polynomial
$$
 V_2=1-\tfrac{2}{1-q}x+\tfrac{2}{(1-q)(1-q^2)}qx^2,
$$
has positive roots
$$
\lambda_1(V_{2})=\tfrac{1-q^2-\sqrt{(1-q)^3(q+1)}}{2q}, \quad \text{and } \lambda_2(V_{2})=\tfrac{1-q^2+\sqrt{(1-q)^3(q+1)}}{2q},
$$
with $0<\lambda_1(V_{2})<\lambda_2(V_{2})$. Moreover, its logarithmic root separation is
$$
\lmesh{V_2}=\frac{1-q^2-\sqrt{(1-q)^3(q+1)}}{1-q^2+\sqrt{(1-q)^3(q+1)}}=\lmesh{U_2}<q.
$$
Suppose that for $n>2$, $V_{n-1}\in\P^q_{n-1}(\R_{>0})$. Denote the zeros of $V_{n-1}$ by $\beta_j:=\lambda_j(V_{n-1})$, since its the logarithmic mesh is less than $q$,
\begin{equation}\label{zerosrn-1}
 0<\beta_1<q^{-1}\beta_1<\cdots<q^{-1}\beta_j<\beta_{j+1}<q^{-1}\beta_{j+1}<\cdots<\beta_{n-1}. 
\end{equation}
Letting $x=\beta_j$ in \eqref{auxiliarpderivative2} we get
 \begin{equation*}
     0=V_{n-1}(\beta_j)=[\D_qV_n](q^{-1}\beta_j)=\frac{V_n(q^{-1}\beta_j)-V_n(\beta_j)}{q^{-1}\beta_j},
 \end{equation*}
 then $V_n(\beta_j)=V_n(q^{-1}\beta_j)$ for  $j=1,\cdots,n-1$. Using Rolle's Theorem we get that the zeros of $V'_n$ satisfy $\lambda_j(V'_n)\in(\beta_j,q^{-1}\beta_j)$. By \eqref{zerosrn-1},
 \begin{equation*}
     [\D V_n](q^{-1}x)\prec V'_n \quad \Longrightarrow \quad  xV'_n\prec [\D V_n](q^{-1}x).
 \end{equation*}
 Note that the sign of the leading coefficient of $V_n$ is equal to $(-1)^n$ and its critical values are positive, then  $V_n$ decreases on $(-\infty,\lambda_1(V'_n))$ and
 since $V_n(0)=1$,  we conclude that all its zeros are strictly positive. Therefore \eqref{eq:rn} satisfies the conditions of Lemma \ref{Jordaaninterlacing}, then we get 
$$
V_n\prec [\D_qV_n](q^{-1}x).
$$
By Remark \ref{remark:lmeshinterlaing}, we conclude $V_n\in\P^q_n(\R_{>0})$.
\end{proof}

\section{Proof of technical lemmas} \label{appendixB}

The goal of this section is to prove Lemma \ref{lem:qdiff.to.hypergeometric}. Although it requires lengthy and cumbersome computations, the proof is, in a certain sense, straightforward once one understands the effect of applying the operators $\D_q$, $\D_{q^2}$, and $\D^2_q$ to powers of $x$.

Recall that for $q\in (0,1)$, our (modified) \textbf{$q$-derivative} operator $\mathfrak{d}_q$ is defined as
\begin{equation*}
    \mathfrak{d}_q f(x) := \begin{cases} 
\frac{f(x) - f(qx)}{x}, & x \neq 0 \\ 
f'(0), & x = 0.
\end{cases}
\end{equation*}
It is readily seen that $\mathfrak{d}_q$ is linear. When applied to powers of $x$ it yields, for $n\in\N$, that
\begin{equation} \label{derivadaxn}
\D_q x^n = (1-q^n)x^{n-1}\quad \text{and} \quad\D_q 1=0.
\end{equation}
Thus, $q$-differentiating $k$ times gives
\begin{equation}
\D_q^{k}[x^n]=x^{n-k} \prod_{i=0}^{k-1}(1-q^{n-i})=x^{n-k} (-1)^kq^{nk-\binom{k}{2}}=\frac{\qraising{q}{q}{n}}{\qraising{q}{q}{n-k}}x^{n-k}\label{eq:qdiff.1}.
\end{equation}
We can use this formula to check that repeatedly applying $\D_{q^2}$ is related to repeatedly applying $\D_{q}$:
\begin{equation}\label{eq:qdiff.2}
    \D_{q^2}^{k}[x^{n}]=x^{n-k} \prod_{i=0}^{k-1}(1-q^{n-i})(1+q^{n-i})= q^{nk-\binom{k}{2}}\qraising{-q^{-n}}{q}{k
    }\D^{k}_{q}[x^{n}].
\end{equation}
 Similarly we notice that applying $\D_q^{2}$ repeatedly is related to $\D_{q^2}$. Specifically, 
\begin{equation}\label{eq:qdiff.3}
    \D_q^{2k}[x^{2n}]=(-1)^kq^{2nk-2\binom{k}{2}-k}\qraising{q^{-2n+1}}{q^2}{k
    }\D^{k}_{q^2}[t^n]\bigg\vert_{t=x^2}.
\end{equation}

The last technical bit that we require is the following formula, a direct consequence of the definition of the $q$-factorials \eqref{eq:def.qfactorial},
\begin{equation}\label{qraising2.bis}
 \qraising{a}{q}{k}(-1)^k q^{-\binom{k}{2}}=  \frac{(-1)^nq^{-\binom{n}{2}}\qraising{a}{q}{n}}{a^{n-k}\qraising{a^{-1} q^{1-n}}{q}{n-k}} .    
\end{equation}

\begin{proof}[Proof of Lemma \ref{lem:qdiff.to.hypergeometric}]
We begin by proving \eqref{eq:qdifferentialoperator1}. The strategy is to expand the left-hand side using \eqref{eq:qdiff.1}, and applying \eqref{qraising2.bis} to each parameter in order to obtain the $q$-hypergeometric polynomial on the left-hand side:
\begin{align*}
\pQq{r}{s}{\bm a}{\bm b}{q}{c\D_q}[x^n] 
&= \sum_{k=0}^{\infty}\frac{\qraising{\bm a}{q}{k}}{\qraising{\bm b}{q}{k}\qraising{q}{q}{k}}(-1)^{(s-r+1)k}q^{(s-r+1)\binom{k}{2}}c^k \D_q^k x^n \\
&\propto \sum_{k=0}^{n}  \left( \frac{b_1\cdots b_sq}{a_1 \cdots a_r} \right)^{n-k} \frac{\qraising{\bm b^{-1} q^{1-n}}{q}{n-k}\qraising{q^{-n}}{q}{n-k} }{{\qraising{\bm a^{-1} q^{1-n}}{q}{n-k}}} \frac{\qraising{q}{q}{n}}{\qraising{q}{q}{n-k}}\frac{x^{n-k}}{c^{n-k}} \\
& \propto  \sum_{k=0}^{n} \frac{\qraising{\bm b^{-1} q^{1-n}}{q}{n-k}\qraising{q^{-n}}{q}{n-k} }{{\qraising{\bm a^{-1} q^{1-n}}{q}{n-k}}\qraising{q}{q}{n-k}} \left( \frac{b_1b_2\cdots b_s}{ca_1a_2 \cdots a_r} qx\right)^{n-k}\\
&\propto \pQq{r+s+1}{r+s}{q^{-n},\bm b^{-1} q^{1-n},\bm 0_{r}}{\bm a^{-1} q^{1-n}, \bm 0_{s}}{q}{ c' qx},
\end{align*}
where $c':= \frac{b_1b_2\dots b_s}{ca_1a_2\dots a_r}$ is as in 
\eqref{eq:constants}. This concludes the proof of \eqref{eq:qdifferentialoperator1}.

To prove \eqref{eq:qdifferentialoperator2}, we will expand the right-hand side and use \eqref{eq:qdiff.2} to reduce it to the previous formula that we just derived.
 \begin{align*}
       \pQq{r}{s}{\bm a }{\bm b}{q}{c\D_{q^2}}[x^{n}] &=\sum_{j=0}^n\frac{\qraising{\bm a}{q}{k}}{\qraising{\bm b}{q}{k}}(-1)^{(s-r+1)k}q^{(s-r+1)\binom{k}{2}}\frac{c^k}{\qraising{q}{q}{k}}\D_{q^2}^k[x^{n}] \\
       &=\sum_{j=0}^n\frac{\qraising{\bm a}{q}{k}\qraising{-q^{-n}}{q}{k}}{\qraising{\bm b}{q}{k}}(-1)^{(s-r)k}q^{(s-r)\binom{k}{2}}\frac{(-q^nc)^k}{\qraising{q}{q}{k}}\D_{q}^k[x^{n}].\\
        &=\pQq{r+1}{s}{\bm a,-q^{-n} }{\bm b}{q}{-q^nc\D_q}[x^{n}]  
    \end{align*}
Equation \eqref{eq:qdifferentialoperator2} then follows from invoking \eqref{eq:qdifferentialoperator1}, and noticing that the dilation constant is given by $qc'$ where $c'= \frac{b_1b_2\dots b_s}{-q^{n}c a_1a_2\dots a_r(-q^{-n})}:= \frac{b_1b_2\dots b_s}{c a_1a_2\dots a_r}$, as in \eqref{eq:constants}.

Finally, we now use \eqref{eq:qdiff.3} and our first formula to prove \eqref{eq:qdifferentialoperator3}. 
 \begin{align*}
         \pQq{r}{s}{\bm a}{\bm b}{q^2}{c\D_q^2}[x^{2n}] 
         &= \sum_{j=0}^n\frac{\qraising{\bm a}{q^2}{k}}{\qraising{\bm b}{q^2}{k}}(-1)^{(s-r+1)k}q^{2(s-r+1)\binom{k}{2}}\frac{c^k\D_q^{2k}}{\qraising{q^2}{q^2}{k}}[x^{2n}]\\
         &=\pQq{r+1}{s}{\bm a,q^{-2n+1}}{\bm b}{q^2}{cq^{2n-1}\D_{q^2}}[t^{n}]\bigg\vert_{t=x^2}\\
        &\propto \pQq{r+s+2}{r+s+1}{q^{-n},\bm b^{-1} q^{2-2n},\bm 0_{r+1}}{\bm a^{-1} q^{2-2n}, q,\bm 0_{s}}{q^2}{ c' q^2 x^2},
    \end{align*}
where in the last line we used \eqref{eq:qdifferentialoperator1}, with $q^2$ instead of $q$, and the dilation constant is $q^2c'$ where
$$c'= \frac{ b_1b_2\dots b_s}{q^{2n-1}c a_1a_2\dots a_rq^{-2n+1}}= \frac{ b_1b_2\dots b_s}{c a_1a_2\dots a_r},$$
as stated.
\end{proof}

The assumption in Lemma \ref{lem:qdiff.to.hypergeometric} that the parameters $a_1,\dots, a_r,b_1,\dots, b_s$ are nonzero can be relaxed. One can informally derive this formula by simply taking $a_l=0$ or $b_l=0$ in Lemma \ref{lem:qdiff.to.hypergeometric}, with the convention that the corresponding term ``$0^{-1}q^{1-n}=\infty$'' on the right-hand side will actually not appear, and adjusting the constant $c'$ by a factor of $q^{1-n}$ (multiplying or dividing) each time we use a zero. We state the result here for completeness.
\begin{lemma}
\label{lem:qdiff.to.hypergeometric.gen}
Consider a constant $c\neq 0$ and tuples of parameters $\bm a=(a_1,\dots,a_{r'})$ and $\bm b =(b_1,\dots,b_{s'})$, where $a_1,\dots, a_{r'},b_1,\dots, b_{s'}$ are non-zero real numbers satisfying restrictions~\eqref{constrainB}. Let $r'',s''$ be non-negative integers and set $r=r'+r''$ and $s=s'+s''$. As before, define the constant
\begin{equation}\label{cor:constants.gen}
    c':= \frac{ b_1b_2\dots b_s}{ca_1a_2\dots a_r}.
\end{equation}
Then the following three formulas hold:
\begin{align}
\pQq{r}{s}{\bm a, \bm 0_i}{\bm b, \bm 0_j}{q}{c\D_q}[x^n] 
&\propto \pQq{r+s'+1}{s+r'}{q^{-n},\bm b^{-1} q^{1-n},\bm 0_{r}}{\bm a^{-1} q^{1-n}, \bm 0_{s}}{q}{ c'q^{(j-i)(1-n)+1}x},
\label{eq:qdifferentialoperator1.gen} \\
\pQq{r}{s}{\bm a, \bm 0_i }{\bm b, \bm 0_j}{q}{c\D_{q^2}}[x^n] 
& \propto\pQq{r+s'+2}{s+r'+1}{q^{-n},\bm b^{-1} q^{1-n},\bm 0_{r+1}}{\bm a^{-1} q^{1-n},-q, \bm 0_{s}}{q}{ c'q^{(j-i)(1-n)+1} x}, \quad \text{and} 
\label{eq:qdifferentialoperator2.gen} \\
\pQq{r}{s}{\bm a, \bm 0_i }{\bm b, \bm 0_j}{q^2}{c\D_{q}^2}[x^{2n}]
&\propto\pQq{r+s'+2}{s+r'+1}{q^{-2n},\bm b^{-1} q^{2-2n},\bm 0_{r+1}}{\bm a^{-1} q^{2-2n},q, \bm 0_{s}}{q^2}{ c'q^{2(j-i)(1-n)+2}x^2},
\label{eq:qdifferentialoperator3.gen}
\end{align}
where $\bm 0_k$ is just a vector of $k$ zeros.
\end{lemma}

We omit the proof, as it follows the same steps of the proof of Lemma \ref{lem:qdiff.to.hypergeometric}, just keeping track of the new terms introduced by the zeros. Alternatively, this result can be derived as a limit of Lemma \ref{lem:qdiff.to.hypergeometric}, by letting some parameters tend to zero.

\begin{remark}
Notice that on the left-hand side of \eqref{eq:qdifferentialoperator1} the vectors of zeros $\bm 0_i$ and $\bm 0_j$ may cancel each other. In practice, we need to consider the case when at least one of $i$ and $j$ is 0. The same goes for the right-hand side with $r$ and $s$. The reason we opted to include both vectors is to provide a unified treatment and avoid distinguishing four different cases in which the zeros might be above or below on each side of the equation. 
\end{remark}

 \section*{Acknowledgments}
The first author was partially supported by Junta de Andaluc\'{\i}a, Spain (research group FQM-229 and Instituto Interuniversitario Carlos I de F\'{\i}sica Te\'orica y Computacional).  The third author was partially supported by AMS-Simons travel grant.

\end{document}